\documentclass[10pt]{article}

\usepackage{amssymb,latexsym,amsmath,enumerate,verbatim,amsfonts,amsthm,cite}
\usepackage{color} 
\usepackage[active]{srcltx}
\usepackage{graphicx}

\def\disp{\displaystyle}

\def\tto{\;{\lower 1pt \hbox{$\rightarrow$}}\kern -10pt
\hbox{\raise 2pt \hbox{$\rightarrow$}}\;}

\def\hat{\widehat}

\def\tilde{\widetilde}
\def\Bar{\overline}
\def\ra{\rangle}
\def\la{\langle}

\def\R{\mathbb{R}}
\def\oR{\overline\mathbb{R}}

\def\st{\stackrel}
\def\oR{\Bar{\R}}

\newtheorem{definition}{Definition}

\newtheorem{theorem}{Theorem}

\newtheorem{remark}{Remark}
\newtheorem{assumption}{Assumption}
\newtheorem{example}{Example}

\def\A{{\mathcal{A}}}

\def\D{{\mathcal{D}}}
\def\I{{\mathcal{I}}}
\def\M{{\mathcal{M}}}
\def\Q{{\mathcal{Q}}}
\def\T{{\mathcal{T}}}
\def\H{{\mathcal{H}}}

\def\cR{{\mathcal{R}}}
\def\R{{\rm I\!R}}

\def\Argmin{\mathop{\rm Arg\,min}}

\title{\sf Global convergence of splitting methods for nonconvex composite optimization}

\author{
Guoyin Li \thanks{Department of Applied
Mathematics, University of New South Wales, Sydney 2052, Australia.
E-mail: {g.li@unsw.edu.au}. This author was partially supported by a research grant from Australian Research Council.}
\and Ting Kei Pong \thanks{Department of Applied Mathematics, the Hong Kong Polytechnic University, Hong Kong. This author was also supported as a PIMS Postdoctoral Fellow at Department of Computer Science,
University of British Columbia, Vancouver, Canada, during the early stage of preparation of this manuscript. E-mail: {tk.pong@polyu.edu.hk}.}
}

\date{Revised Version: September 29, 2015}

\begin{document}
  \maketitle

\begin{abstract}
  We consider the problem of minimizing the sum of a smooth function $h$ with a bounded Hessian, and a nonsmooth function.
   We assume that the latter function is a composition of a proper closed function $P$ and a surjective linear map $\M$, with the proximal mappings of $\tau P$, $\tau > 0$, simple to compute. This problem is nonconvex in general and encompasses many important applications in engineering and machine learning. In this
   paper, we examined two types of splitting methods for solving this nonconvex optimization problem: alternating direction method of multipliers and proximal gradient algorithm. For the direct adaptation of the alternating direction method of multipliers, we show that, if the penalty parameter is chosen sufficiently large and the sequence generated has a cluster point, then it gives a stationary point of the nonconvex problem. We also establish convergence of the whole sequence under an additional assumption that the functions $h$ and $P$ are semi-algebraic. Furthermore, we give simple sufficient conditions to guarantee boundedness of the sequence generated. These conditions can be satisfied for a wide range of applications including the least squares problem with the $\ell_{1/2}$ regularization. Finally, when $\M$ is the identity so that the proximal gradient algorithm can be efficiently applied, we show that any cluster point is stationary under a slightly more flexible constant step-size rule than what is known in the literature for a nonconvex $h$.
\end{abstract}

\section{Introduction}

In this paper, we consider the following optimization problem:
\begin{equation}\label{P1}
  \begin{array}{rl}
    \min\limits_x & h(x) + P(\M x),
  \end{array}
\end{equation}
where $\M$ is a linear map from $\R^n$ to $\R^m$, $P$ is a proper closed function on $\R^m$ and $h$ is twice continuously differentiable on $\R^n$ with a bounded Hessian. 
We also assume that the proximal (set-valued) mappings
\[
u\mapsto \Argmin_y\left\{\tau P(y) + \frac12\|y - u\|^2\right\}
\]
are well-defined and are simple to compute for all $u$ and for any $\tau > 0$. Here, $\Argmin$ denotes the set of minimizers, and the simplicity is understood in the sense that {\em at least one} element of the set of minimizers can be obtained efficiently. Concrete examples of such $P$ that arise in applications include functions listed in \cite[Table~1]{GHLYZ13}, the $\ell_{1/2}$ regularization \cite{ZengLinWangXu14},
the $\ell_0$ regularization, and the
indicator functions of the set of vectors with cardinality at most $s$ \cite{BluDav08}, matrices with rank at most $r$ and $s$-sparse vectors in simplex \cite{KBCK13}, etc. Moreover, for a large class of nonconvex functions, a general algorithm has been proposed recently in \cite{HareSaga09} for computing the proximal mapping.

The model problem \eqref{P1} with $h$ and $P$ satisfying the above assumptions encompasses many important applications in engineering and machine learning; see, for example, \cite{BluDav08,GHLYZ13,LiV09,CaT05,CaRe09}. In particular, many sparse learning problems are in the form of \eqref{P1} with $h$ being a loss function, $\M$ being the identity map and $P$ being a regularizer; see, for example, \cite{BluDav08} for the use of the $\ell_0$ norm as a regularizer, \cite{CaT05} for the use of the $\ell_1$ norm, \cite{CaRe09} for the use of the nuclear norm, and \cite{GHLYZ13} and the references therein for the use of various continuous difference-of-convex functions with simple proximal mappings. For the case when $\M$ is not the identity map, an application in stochastic realization where $h$ is a least squares loss function, $P$ is the rank function and $\M$ is the linear map that takes the variable $x$ into a block Hankel matrix was discussed in \cite[Section~II]{LiV09}.

When $\M$ is the identity map, the proximal gradient algorithm \cite{FukuMine81,Gabay83,Tse91} (also known as forward-backward splitting algorithm) can be applied whose subproblem involves a computation of the proximal mapping of $\tau P$ for some $\tau > 0$. It is known that when $h$ and $P$ are convex, the sequence generated from this algorithm is convergent to a globally optimal solution if the step-size is chosen from $(0,\frac2L)$, where $L$ is any number larger than the Lipschitz continuity modulus of $\nabla h$. For nonconvex $h$ and $P$, the step-size can be chosen from $(0,\frac1L)$ so that any cluster point of the sequence generated is stationary \cite[Proposition~2.3]{BreLo09} (see Section~\ref{sec2} for the definition of stationary points), and convergence of the whole sequence is guaranteed if the sequence generated is bounded and $h + P$ satisfies the Kurdyka-{\L}ojasiewicz (KL) property \cite[Theorem~5.1, Remark~5.2(a)]{AtBoSv13}. On the other hand, when $\M$ is a general linear map so that the computation of the
proximal mapping of $\tau P\circ \M$, $\tau > 0$, is not necessarily simple, the proximal gradient algorithm cannot be applied efficiently. In the case when $h$ and $P$ are both convex, one feasible approach is to apply the alternating direction method of multipliers (ADMM) \cite{EckB92,FGlowinski83,GaM76}. This has been widely used recently; see, for example \cite{WangYang08,WenGY10,YaZ09,CYY11,CHY12}. While it is tempting to directly apply the ADMM to the nonconvex problem \eqref{P1}, convergence has only been shown under specific assumptions. In particular, in \cite{WenPengLiuBaiSun13}, the authors studied an application that can be modeled as \eqref{P1} with $h=0$, $P$ being some risk measures and $\M$ typically being an injective linear map coming from data. They showed that any cluster point gives a stationary point, assuming square summability of the successive changes in the dual iterates. More recently, in \cite{AmesHong14}, the authors considered the case when $h$ is a nonconvex quadratic and $P$
is the sum of the $\
\ell_1$ norm and the indicator function of the Euclidean norm ball. They showed that if the penalty parameter is chosen sufficiently large (with an explicit lower bound) and the dual iterates satisfy a particular assumption, then any cluster point gives a stationary point. In particular, their assumption is satisfied if $\M$ is surjective.

Motivated by the findings in \cite{AmesHong14}, in this paper, we focus on the case when $\M$ is surjective and consider both the ADMM (for a general surjective $\M$) and the proximal gradient algorithm (for $\M$ being the identity). The contributions
of this paper are as follows:
\begin{itemize}
 \item First, we characterize cluster points of the sequence generated from the ADMM. In particular, we show that if the (fixed) penalty parameter in the ADMM is chosen sufficiently large (with a computable lower bound), and a cluster point of the sequence generated exists, then it gives a stationary point of problem \eqref{P1}.

 Moreover, our analysis allows replacing $h$ in the ADMM subproblems by its local quadratic approximations so that in each iteration of this variant, the subproblems only involve computing the proximal mapping of $\tau P$ for some $\tau > 0$ and
solving an unconstrained convex quadratic minimization problem. Furthermore, we also give simple sufficient conditions to guarantee the boundedness of the sequence generated. These conditions are satisfied in a wide range of applications; see Examples~\ref{examplenew:3}, \ref{examplenew:1} and \ref{examplenew:2}.

\item Second, under the additional assumption that $h$ and $P$ are semi-algebraic functions, we show that if a cluster point of the sequence generated from the ADMM exists, it is actually convergent.
Our assumption on semi-algebraicity not only can be easily verified or recognized, but also covers a broad class of optimization problems such as problems involving quadratic functions, polyhedral norms and the cardinality function.

\item Third, we give a concrete 2-dimensional  counterexample in Example~\ref{ex7:nonconverge} showing that the ADMM can be divergent when $\M$ is assumed to be injective (instead of surjective).

\item Finally, for the particular case when $\M$ equals the identity map, we show that the proximal gradient algorithm can be applied with a slightly more flexible step-size rule when $h$ is nonconvex (see Theorem~\ref{prop:prox} for the precise statement).
\end{itemize}

The rest of the paper is organized as follows. We discuss notation and preliminary materials in the next section. Convergence of the ADMM is analyzed in Section~\ref{sec:Msur}, and Section~\ref{sec:MI} is devoted to the analysis of the proximal gradient algorithm. Some numerical results are presented in Section~\ref{sec:num} to illustrate the algorithms. We give concluding remarks and discuss future research directions in Section~\ref{sec:con}.

\section{Notation and preliminaries}\label{sec2}

We denote the $n$-dimensional Euclidean space as $\R^n$, and use $\langle\cdot,\cdot\rangle$
to denote the inner product and $\|\cdot\|$ to denote the norm induced from the inner product.
Linear maps are denoted by scripted letters. The identity map is denoted by $\I$.
For a linear map $\M$, $\M^*$ denotes the adjoint linear map with respect
to the inner product and
$\|\M\|$ is the induced operator norm of $\M$. A linear self-map $\T$ is called symmetric if $\T = \T^*$.
For a symmetric linear self-map $\T$, we use $\|\cdot\|_\T^2$ to denote its induced quadratic form given by
$\|x\|_\T^2 = \langle x,\T x\rangle$ for all $x$, and use $\lambda_{\max}$ (resp., $\lambda_{\min}$) to denote the maximum (resp., minimum) eigenvalue of $\T$.
A symmetric linear self-map $\T$ is called positive semidefinite, denoted by $\T\succeq 0$ (resp., positive definite, $\T\succ 0$) if
$\|x\|_\T^2\ge 0$ (resp., $\|x\|_\T^2> 0$) for all nonzero $x$. For two symmetric linear self-maps $\T_1$ and $\T_2$, we use $\T_1\succeq \T_2$ (resp., $\T_1\succ \T_2$) to denote $\T_1 - \T_2\succeq 0$ (resp., $\T_1 - \T_2\succ 0$).

An extended-real-valued function $f$ is called proper if it is finite somewhere and never equals $-\infty$. Such a function is called closed if it is lower semicontinuous.
Given a proper function $f:\R^n \to\oR:=(-\infty,\infty]$, we use the symbol $z\st{f}{\to}x$ to indicate $z\to x$ and $f(z)\to f(x)$. The domain of $f$ is denoted by ${\rm dom}f$ and is defined as ${\rm dom}f=\{x \in \R^n: f(x)<+\infty\}$. Our basic {\em subdifferential} of $f$ at $x\in\mathrm{dom}\,f$ (known also as the limiting subdifferential) is defined by (see, for example, \cite[Definition~8.3]{Rock98})
\begin{equation}\label{ls}
\partial f(x):=\left\{v\in\R^n :\;\exists x^t\st{f}{\to}x,\;v^t\to v\;\mbox{ with }\disp\liminf_{z\to x^t}\frac{f(z)-f(x^t)-\la v^t,z-x^t\ra}{\|z-x^t\|}\ge 0\mbox{ for each }t\right\}.
\end{equation}
It follows immediately from the above definition that this subdifferential has the following robustness property:
\begin{equation}\label{outersemi}
  \left\{v\in \R^n:\; \exists x^t \st{f}{\to}x,\; v^t\to v\;, v^t\in \partial f(x^t)\right\} \subseteq \partial f(x).
\end{equation}
For a convex function $f$ the subdifferential \eqref{ls} reduces to the classical subdifferential in convex analysis (see, for example, \cite[Theorem 1.93]{Boris})
\begin{eqnarray*}
\partial f(x)=\left\{v\in \R^n :\;\langle v,z-x\rangle\le f(z)-f(x)\ \ \forall \ z\in\R^n\right\}.
\end{eqnarray*}
Moreover, for a continuously differentiable function $f$, the subdifferential \eqref{ls} reduces to the derivative of $f$ denoted by $\nabla f$. For a function $f$ with more than one group
of variables, we use $\partial_x f$ (resp., $\nabla_x f$) to denote the subdifferential (resp., derivative) of $f$ with respect to the variable $x$.
Furthermore, we write ${\rm dom}\,\partial f = \{x\in \R^n:\; \partial f(x)\neq \emptyset\}$.

In general, the subdifferential set \eqref{ls} can be nonconvex (e.g., for $f(x)=-|x|$ at $0\in\R$) while $\partial f$
enjoys comprehensive calculus rules based on {\em variational/extremal principles} of variational analysis \cite{Rock98}.
In particular, when $\M$ is a surjective linear map, using \cite[Exercise~8.8(c)]{Rock98} and \cite[Exercise~10.7]{Rock98}, we see that
\[
\partial (h + P\circ \M)(x) = \nabla h( x) + \M^* \partial P(\M x)
\]
for any $x\in {\rm dom}(P\circ \M)$. Hence, at an optimal solution $\bar x$, the following necessary optimality condition always holds:
\begin{equation}\label{optcon2}
  0 \in \partial (h + P\circ \M)(\bar x) = \nabla h(\bar x) + \M^* \partial P(\M\bar x).
\end{equation}
Throughout this paper, we say that $\tilde x$ is a stationary point of \eqref{P1}
if $\tilde x$ satisfies \eqref{optcon2} in place of $\bar x$.

For a continuously differentiable function $\phi$ on $\R^n$, the Bregman distance $D_\phi$ is defined as
\begin{equation*}
  D_\phi(x_1,x_2) := \phi(x_1) - \phi(x_2) - \langle \nabla \phi(x_2), x_1 - x_2\rangle
\end{equation*}
for any $x_1$, $x_2\in \R^n$. If $\phi$ is twice continuously differentiable and there exists $\Q$ so that the Hessian $\nabla^2\phi$ satisfies $[\nabla^2\phi(x)]^2\preceq \Q$ for all $x$, then for any $x_1$ and $x_2$ in $\R^n$, we have
\begin{equation}\label{normbd}
\begin{split}
  &\|\nabla \phi(x_1) - \nabla \phi(x_2)\|^2  = \left\|\int_0^1\nabla^2\phi(x_2 + t(x_1 - x_2))\cdot[x_1 - x_2] dt\right\|^2\\
  &\le \left(\int_0^1\left\|\nabla^2\phi(x_2 + t(x_1 - x_2))\cdot[x_1 - x_2]\right\| dt\right)^2\\
  & = \left(\int_0^1\sqrt{\langle x_1-x_2,[\nabla^2\phi(x_2 + t(x_1 - x_2))]^2\cdot[x_1 - x_2]\rangle }dt\right)^2 \le \|x_1 - x_2\|^2_{\Q}.
\end{split}
\end{equation}
On the other hand, if there exists $\Q$ so that $\nabla^2\phi(x)\succeq \Q$ for all $x$, then
\begin{equation}\label{Bregman}
\begin{split}
  & D_\phi(x_1,x_2) = \int_0^1\langle \nabla \phi(x_2 + t(x_1 - x_2)) - \nabla \phi(x_2), x_1 - x_2\rangle dt\\
  & = \int_0^1\int_0^1 t\langle x_1 - x_2, \nabla^2 \phi(x_2 + st(x_1 - x_2))\cdot[x_1 - x_2]\rangle ds \ dt \ge \frac12 \|x_1 - x_2\|^2_{\Q}
\end{split}
\end{equation}
for any $x_1$ and $x_2$ in $\R^n$.


A semi-algebraic set $S\subseteq \R^n$
is a finite union of sets of the form
\[\{x \in \R^n: h_1(x) = \cdots = h_k(x) = 0,g_1(x) < 0,\ldots,g_l(x) < 0\},
\]
where $h_1,\ldots, h_k$ and $g_1,\ldots,g_l$ are polynomials with real coefficients in $n$ variables. In other words, $S$ is a union of finitely
many sets, each defined by finitely many polynomial equalities and strict inequalities. A map
$F:\R^n \rightarrow \R$ is semi-algebraic if ${\rm gph}F \in \R^{n+1}$ is a semi-algebraic set. Semi-algebraic sets and semi-algebraic mappings
enjoy many nice structural properties.
One important property which we will use later on is the Kurdyka-{\L}ojasiewicz (KL) property. 

\begin{definition}\label{def:KL} {\bf (KL property \& KL function)}
  A proper  function $f$ is said to have the Kurdyka-{\L}ojasiewicz (KL) property at $\hat x\in {\rm dom}\,\partial f$ if there exist $\eta\in (0,\infty]$, a neighborhood $V$ of $\hat x$ and a continuous concave function $\varphi:[0,\eta)\rightarrow {\mathbb{R}}_+$ such that:
  \begin{enumerate}[{\rm (i)}]
    \item $\varphi(0) = 0$ and $\varphi$ is continuously differentiable on $(0,\eta)$ with positive derivatives;
    \item for all $x\in V$ satisfying $f(\hat x)< f(x) < f(\hat x) + \eta$, it holds that
    \begin{equation*}
      \varphi'(f(x) - f(\hat x))\,{\rm dist}(0,\partial f(x))\ge 1.
    \end{equation*}
  \end{enumerate}
  A proper closed function $f$ satisfying the KL property at all points in ${\rm dom}\,\partial f$ is called a KL function.
\end{definition}


It is known that a proper closed semi-algebraic function is a KL function as such a function satisfies the KL property for all points in ${\rm dom}\,\partial f$ with $\varphi(s) = cs^{1-\theta}$ for some $\theta\in [0,1)$ and some $c>0$  (for example, see
\cite[Section~4.3]{AtBoReSo10}; further discussion can be found in \cite[Corollary~16]{BolDanLewShi07} and \cite[Section~2]{BolDanLew07}).
%

\section{Alternating direction method of multipliers 
}\label{sec:Msur}

In this section,
we study the alternating direction method of multipliers for finding a stationary point of \eqref{P1}.
To describe the algorithm, we first reformulate \eqref{P1} as
\begin{equation*}
  \begin{array}{rl}
    \displaystyle \min_{x,y} & h(x) + P(y)\\
    {\rm s.t.} & y = \M x,
  \end{array}
\end{equation*}
to decouple the linear map and the nonsmooth part. Recall that the augmented Lagrangian function for the above problem is defined, for each $\beta>0$, as:
\begin{equation*}
  L_\beta(x,y,z) := h(x) + P(y) - \langle z,\M x - y\rangle + \frac{\beta}{2}\|\M x - y\|^2.
\end{equation*}
Our algorithm is then presented as follows:\\

\vspace{.1 in}
\fbox{\parbox{5.7 in}{
\begin{description}
\item {\bf Proximal ADMM}

\item[Step 0.] Input $(x^0, z^0)$, $\beta > 0$ and a twice continuously differentiable convex function $\phi(x)$.

\item[Step 1.] Set
\begin{equation}\label{scheme}
\left\{
\begin{split}
&y^{t+1}\in\Argmin_{y} L_\beta(x^t,y,z^t), \\
&x^{t+1}\in\Argmin_{x} \{L_\beta(x,y^{t+1},z^t) + D_\phi(x,x^t)\},\\
&z^{t+1}=z^t-\beta  (\M x^{t+1}- y^{t+1}).
\end{split}
\right.
\end{equation}

\item[Step 2.] If a termination criterion is not met, go to Step 1.
\end{description}
}}
\vspace{.1 in}\\

Notice that the first subproblem is essentially computing the proximal mapping of $\tau P$ for some $\tau > 0$. The above
algorithm is called the proximal ADMM since,
in the second subproblem, we allow a proximal term $D_\phi$ and hence a choice of $\phi$ to simplify this subproblem. If $\phi=0$, then this algorithm reduces to the usual
ADMM described in, for example, \cite{EckB92}. For other popular non-trivial choices of $\phi$, see Remark~\ref{rem0} below.

We next study global convergence of the above algorithm under suitable assumptions. Specifically, we consider the following assumption.
\begin{assumption}\label{assumption}
  \begin{enumerate}[{\rm (i)}]
    \item $\M\M^*\succeq \sigma \I$ for some $\sigma > 0$; and there exist $\Q_1$, $\Q_2$ such that for all $x$,
    $\Q_1\succeq \nabla^2 h(x)\succeq \Q_2$.
    \item $\beta>0$ and $\phi$ are chosen so that
    \begin{itemize}
      \item there exist $\T_1\succeq \T_2 \succeq 0$ so that $\T^2_1 \succeq [\nabla^2\phi(x)]^2\succeq \T^2_2$ for all $x$;
      \item $\Q_2 + \beta \M^*\M + \T_2\succeq \delta \I$ for some $\delta > 0$;
      \item with $\Q_3\succeq [\nabla^2 h(x) + \nabla^2 \phi(x)]^2$ for all $x$, there exists $\gamma \in (0,1)$ so that
      \begin{equation*}
      \delta \I + \T_2 \succ \frac{2}{\sigma\beta}\H_\gamma, \ \ {\rm where}\ \ \H_\gamma := \left(\frac{1}{\gamma}\Q_3 + \frac{1}{1-\gamma}{\T}_1^2\right).
      \end{equation*}

    \end{itemize}
  \end{enumerate}
\end{assumption}
 \begin{remark}{\bf (Comments on Assumption 1)}\label{rem0}
  Point (i) says $\M$ is surjective.
  The first and second points in (ii) would be satisfied if $\phi(x)$ is chosen to be $\frac{L}{2}\|x\|^2 - h(x)$,
  where $L$ is at least as large as the Lipschitz continuity modulus of $\nabla h(x)$. In this case, one can pick $\T_1 = 2L \I$ and $\T_2 = 0$. This choice is of particular interest
  since it simplifies the $x$-update in \eqref{scheme} to a convex quadratic programming problem; see \cite[Section~2.1]{WangBan13}.
  Indeed, under this choice, we have
  \[
  D_\phi(x,x^t) = \frac{L}2 \|x - x^t\|^2 - h(x) + h(x^t) + \langle \nabla h(x^t),x-x^t\rangle,
  \]
  and hence the second subproblem becomes
  \[
    \min\limits_x\ \frac{L}2 \|x - x^t\|^2 + \langle \nabla h(x^t) - \M^*z^t,x-x^t\rangle + \frac{\beta}{2}\|\M x - y^{t+1}\|^2.
  \]
  Finally, point 3 in (ii) can always be enforced by picking $\beta$ sufficiently large if $\phi$, $\T_1$ and $\T_2$, are chosen
  independently of $\beta$. In addition, in the case where $\T_1 = 0$ and hence $\T_2 = 0$, it is not hard to show that the requirement that $\delta \I + \T_2 \succ \frac{2}{\sigma\beta}\H_\gamma$ for some $\gamma \in (0,1)$ is indeed equivalent to imposing $\delta \I \succ \frac{2}{\sigma\beta}\Q_3$.
  \end{remark}
%

Before stating our convergence results, we note
first that from the optimality conditions, the iterates generated satisfy
\begin{equation}\label{eq1}
\begin{split}
  0& \in \partial P(y^{t+1}) + z^t - \beta(\M x^t - y^{t+1}),\\
  0& = \nabla h(x^{t+1}) - \M^* z^t + \beta\M^*(\M x^{t+1} - y^{t+1}) + (\nabla \phi(x^{t+1}) - \nabla \phi(x^t)).
\end{split}
\end{equation}
Hence, if
\begin{equation}\label{eq2:limit}
  \lim_{t \rightarrow \infty} \|y^{t+1} - y^t\|^2 + \|x^{t+1} - x^t\|^2 + \|z^{t+1} - z^t\|^2 = 0,
\end{equation}
and if for a cluster point $(x^*,y^*,z^*)$ of the sequence $\{(x^t,y^t,z^t)\}$, we have
\begin{equation}\label{eq:Plimit}
\lim_{i\rightarrow\infty}P(y^{t_i+1}) = P(y^*)
\end{equation}
along a convergent subsequence $\{(x^{t_i},y^{t_i},z^{t_i})\}$ that converges to $(x^*,y^*,z^*)$,
then $x^*$ is a stationary point of \eqref{P1}. To see this, notice from \eqref{eq1} and the definition of $z^{t+1}$ that
\begin{equation}\label{eq:iterop}
\left\{
\begin{split}
  &-z^{t+1} - \beta \M(x^{t+1} - x^t) \in \partial P(y^{t+1}),\\
  &\nabla h(x^{t+1}) - \M^* z^{t+1} = -\nabla \phi(x^{t+1}) + \nabla \phi(x^t),\\
  &\M x^{t+1} - y^{t+1} = \frac1\beta (z^t - z^{t+1}).
\end{split} \right.
\end{equation}
Passing to the limit in \eqref{eq:iterop} along the subsequence $\{(x^{t_i},y^{t_i},z^{t_i})\}$ and invoking \eqref{eq2:limit}, \eqref{eq:Plimit} and \eqref{outersemi}, it follows that
\begin{equation}\label{eq:optcon}
  \nabla h(x^*) = \M^* z^*,\ \ -z^* \in \partial P(y^*),\ \ y^* = \M x^*.
\end{equation}
In particular, $x^*$ is a stationary point of the model problem \eqref{P1}. 

We now state our global convergence result. Our first conclusion
establishes \eqref{eq2:limit} under Assumption~\ref{assumption}, and so, any cluster point of the sequence generated from the proximal ADMM produces a stationary point of our model problem \eqref{P1}
such that \eqref{eq:optcon} holds. In the case where $h$ is a nonconvex quadratic function
with a negative semi-definite Hessian matrix  and $P$
is the sum of the $\ell_1$ norm and the indicator function of the Euclidean norm ball, the convergence of the ADMM (i.e., proximal ADMM with $\phi = 0$) was established in \cite{AmesHong14}. Our convergence analysis below follows the recent work in \cite[Section~3.3]{AmesHong14} and \cite{WenPengLiuBaiSun13}. Specifically, we follow the idea in \cite{WenPengLiuBaiSun13} to study the behavior of the augmented Lagrangian function along the sequence generated from the proximal ADMM; we note that this was subsequently also used in \cite[Section~3.3]{AmesHong14}. We then bound the changes in $\{z^t\}$ by those of $\{x^t\}$, following the brilliant observation in \cite[Section~3.3]{AmesHong14}
that the changes in the dual iterates can be controlled by the changes in the primal iterates that correspond to the quadratic in their objective. However, we would like to point out two major modifications: {\bf (i)} The proof in \cite[Section~3.3]{AmesHong14} cannot be directly applied because our subproblem corresponding to the $y$-update is not convex due to the possible nonconvexity of $P$. Our analysis is also complicated by the introduction of the proximal term. {\bf (ii)} Using the special structure of their problem, the authors in \cite[Section~3.3]{AmesHong14} established that the augmented Lagrangian for their problem is uniformly bounded below along the sequence generated from their ADMM. In contrast, we {\em assume} existence of cluster points in our convergence analysis below and will discuss sufficient conditions for such an assumption in Theorem \ref{thm:boundedness2}. On the other hand, we have to point out that although our sufficient conditions for boundedness of sequence are general enough to cover a wide range of applications, they do {\em not} cover the particular problem studied in \cite{AmesHong14}.

Our second conclusion, which is new in the literature studying convergence of ADMM in the nonconvex scenarios, states that if the algorithm is suitably initialized, we can get a strict improvement in the objective values. In particular, if suitably initialized, one will not end up with a stationary point with a larger objective value.

\begin{theorem}\label{thm:main}
  Suppose that Assumption~\ref{assumption} holds. Then we have the following results.
  \begin{enumerate}[{\rm (i)}]
    \item {\bf (Global subsequential convergence)} If the sequence $\{(x^t,y^t,z^t)\}$ generated from the proximal ADMM has a cluster point $(x^*,y^*,z^*)$, then \eqref{eq2:limit} holds. Moreover, $x^*$ is a stationary point of \eqref{P1} such that \eqref{eq:optcon} holds.
    \item {\bf (Strict improvement in objective values)} Suppose that the algorithm is initialized at
  a non-stationary $x^0$ with $h(x^0) + P(\M x^0) < \infty$, and $z^0$ satisfying $\M^* z^0 = \nabla h(x^0)$. Then for any cluster point $(x^*,y^*,z^*)$ of the sequence $\{(x^t,y^t,z^t)\}$, if exists, we have
  \[
  h(x^*) + P(\M x^*) < h(x^0) + P(\M x^0).
  \]
  \end{enumerate}
\end{theorem}

\begin{remark}\label{rem1}
The proximal ADMM does not necessarily guarantee that the objective value of \eqref{P1} is decreasing along the sequence $\{x^t\}$ generated. However, under the assumptions in Theorem~\ref{thm:main}, any cluster point of the sequence generated from the proximal ADMM improves the starting (non-stationary) objective value.

  We now describe one way of choosing the initialization as suggested in (ii) when $P$ is nonconvex. In this case, it is common to approximate $P$ by a proper closed convex function $\tilde P$ and obtain a relaxation to \eqref{P1}, i.e.,
 \[
 \min_x \ h(x) + \tilde P(\M x).
 \]
 Then any stationary point $\tilde x$ of this relaxed problem, if exists, satisfies $-\nabla h(\tilde x)\in \M^*\partial \tilde P(\M \tilde x)$. Thus, if $P(\M\tilde x) < \infty$, then one can initialize the proximal ADMM by taking $x^0 =\tilde x$ and $z^0\in-\partial \tilde P(\M \tilde x)$ with $\nabla h(\tilde x) = \M^*z^0$, so that the conditions in (ii) are satisfied.
\end{remark}

\begin{proof}

  We start by showing that \eqref{eq2:limit} holds.
  First, observe from the second relation in \eqref{eq:iterop} that
  \begin{equation}\label{Mzhrelation}
    \M^* z^{t+1} = \nabla h(x^{t+1}) + \nabla \phi(x^{t+1}) - \nabla \phi(x^t).
  \end{equation}
  Consequently, we have
  \begin{equation*}
    \M^* (z^{t+1} - z^t) = \nabla h(x^{t+1}) - \nabla h(x^t) + (\nabla \phi(x^{t+1}) - \nabla \phi(x^t)) - (\nabla \phi(x^{t}) - \nabla \phi(x^{t-1})).
  \end{equation*}
  Taking norm on both sides, squaring and making use of (i) in Assumption~\ref{assumption}, we obtain further that
  \begin{equation}\label{eq6}
    \begin{split}
      & \sigma\|z^{t+1} - z^t\|^2  \le \|\M^* (z^{t+1} - z^t)\|^2 \\
      & = \|\nabla h(x^{t+1}) - \nabla h(x^t) + (\nabla \phi(x^{t+1}) - \nabla \phi(x^t)) - (\nabla \phi(x^{t}) - \nabla \phi(x^{t-1}))\|^2\\
      & \le \frac{1}{\gamma}\|\nabla h(x^{t+1}) - \nabla h(x^t) + \nabla \phi(x^{t+1}) - \nabla \phi(x^t)\|^2 + \frac1{1-\gamma}\|\nabla \phi(x^{t}) - \nabla \phi(x^{t-1})\|^2 \\
      & \le \frac{1}{\gamma}\|x^{t+1}-x^t\|_{\Q_3}^2 + \frac1{1-\gamma}\|x^{t} - x^{t-1}\|_{\T_1^2}^2,
    \end{split}
  \end{equation}
  where $\gamma\in (0,1)$ is defined in point 3 in (ii) of Assumption~\ref{assumption}, and we made use of the relation $\|a + b\|^2 \le \frac{1}{\gamma}\|a\|^2 + \frac1{1 - \gamma} \|b\|^2$ for the first inequality, while the last inequality follows from points 1 and 3 in (ii) of Assumption~\ref{assumption}, and \eqref{normbd}.
  On the other hand, from the definition of $z^{t+1}$, we have
  \begin{equation*}
    y^{t+1} = \M x^{t+1} + \frac{1}{\beta}(z^{t+1} - z^t),
  \end{equation*}
  which implies
  \begin{equation}\label{eq7}
    \|y^{t+1} - y^t\| \le \|\M(x^{t+1}-x^t)\| + \frac{1}{\beta}\|z^{t+1} - z^t\| + \frac{1}{\beta}\|z^t - z^{t-1}\|.
  \end{equation}
  In view of \eqref{eq6} and \eqref{eq7}, to establish \eqref{eq2:limit}, it suffices to show that
  \begin{equation}\label{eq2:limit2}
    \lim_{t \rightarrow \infty} \|x^{t+1} - x^t\| = 0.
  \end{equation}

  We now prove \eqref{eq2:limit2}. We start by noting that
  \begin{equation}\label{term1}
  \begin{split}
    &L_\beta(x^{t+1},y^{t+1},z^{t+1}) - L_\beta(x^{t+1},y^{t+1},z^t) = - (z^{t+1}-z^t)^T(\M x^{t+1} - y^{t+1})\\
    & = \frac{1}{\beta}\|z^{t+1}-z^t\|^2 \le \frac{1}{\sigma\beta}(\|x^{t+1}-x^t\|_{\frac1\gamma \Q_3}^2 + \|x^{t} - x^{t-1}\|_{\frac1{1 - \gamma} \T_1^2}^2).
  \end{split}
  \end{equation}
  Next, recall from \cite[Page~553,~Ex.17]{HJ91} that the operation of taking positive square root preserves the positive semidefinite ordering. Thus, point 1 in (ii) of Assumption~\ref{assumption} implies that $\nabla^2\phi(x)\succeq \T_2$ for all $x$. From this and
  point 2 in (ii) of Assumption~\ref{assumption}, we see further that the function $x\mapsto L_\beta(x,y^{t+1},z^t) + D_\phi(x,x^t)$
  is strongly convex with modulus at least $\delta$. Using this, the definition
  of $x^{t+1}$ (as a minimizer) and \eqref{Bregman}, we have
  \begin{equation}\label{term2}
    \begin{split}
      L_\beta(x^{t+1},y^{t+1},z^t) - L_\beta(x^t,y^{t+1},z^t) \le -\frac{\delta}{2}\|x^{t+1}-x^t\|^2 - \frac12\|x^{t+1}-x^t\|_{\T_2}^2.
    \end{split}
  \end{equation}
  Moreover, using the definition of $y^{t+1}$ as a minimizer, we have
  \begin{equation}\label{term3}
    \begin{split}
      L_\beta(x^t,y^{t+1},z^t) - L_\beta(x^t,y^t,z^t) \le 0.
    \end{split}
  \end{equation}
  Summing \eqref{term1}, \eqref{term2} and \eqref{term3}, we obtain that
  \begin{equation}\label{eq:smallrelation}
  \begin{split}
  &L_\beta(x^{t+1},y^{t+1},z^{t+1}) - L_\beta(x^t,y^t,z^t) \\
  &\le\frac12\|x^{t+1}-x^t\|_{\frac{2}{\sigma\beta\gamma}\Q_3-\delta \I -\T_2}^2 +\frac12\|x^{t} - x^{t-1}\|_{\frac{2}{\sigma\beta(1 - \gamma)}\T_1^2}^2.
  \end{split}
  \end{equation}
  Summing the above relation from $t=M,...,N-1$ with $M \ge 1$, we see that
  \begin{equation}\label{eq:bigrelation}
  \begin{split}
  &    L_\beta(x^N,y^N,z^N) - L_\beta(x^M,y^M,z^M)\\
  &\le\frac12\sum_{t=M}^{N-1}\|x^{t+1}-x^t\|_{\frac{2}{\sigma\beta\gamma}\Q_3-\delta \I -\T_2}^2 +\frac12\sum_{t=M}^{N-1}\|x^{t} - x^{t-1}\|_{\frac{2}{\sigma\beta(1-\gamma)}\T_1^2}^2\\
  &=\frac12\sum_{t=M}^{N-1}\|x^{t+1}-x^t\|_{\frac{2}{\sigma\beta\gamma}\Q_3-\delta \I -\T_2}^2 +\frac12\sum_{t=M-1}^{N-2}\|x^{t+1} - x^{t}\|_{\frac{2}{\sigma\beta(1-\gamma)}\T_1^2}^2\\
  &=\frac12\sum_{t=M}^{N-2}\|x^{t+1}-x^t\|_{\frac{2}{\sigma\beta}\H_\gamma-\delta \I-\T_2}^2 + \frac12\|x^N-x^{N-1}\|_{\frac{2}{\sigma\beta\gamma}\Q_3-\delta \I -\T_2}^2 +\frac12\|x^{M} - x^{M-1}\|_{\frac{2}{\sigma\beta(1-\gamma)}\T_1^2}^2\\
  &\le -\frac12\sum_{t=M}^{N-2}\|x^{t+1}-x^t\|_{\cR}^2 +\frac12\|x^{M} - x^{M-1}\|_{\frac{2}{\sigma\beta(1-\gamma)}\T_1^2}^2,
  \end{split}
  \end{equation}
  where $\cR:= \delta \I+\T_2 - \frac{2}{\sigma\beta}\H_\gamma \succ 0$ due to point 3 in (ii) of Assumption~\ref{assumption}; and the last inequality follows from $\delta \I+\T_2 - \frac{2}{\sigma\beta\gamma}Q_3 \succeq \cR \succ 0$.

  Now, suppose that $(x^*,y^*,z^*)$ is a cluster point of the sequence $\{(x^t,y^t,z^t)\}$ and consider a convergent subsequence, i.e.,
  \begin{equation}\label{subcon}
  \lim_{i\rightarrow \infty}(x^{t_i},y^{t_i},z^{t_i}) = (x^*,y^*,z^*).
  \end{equation}
  From lower semicontinuity of $L$, we see that
  \begin{equation}\label{eq:limitL}
  \liminf_{i\rightarrow \infty}L_\beta(x^{t_i},y^{t_i},z^{t_i})
  \ge h(x^*) + P(y^*) - \langle z^*,\M x^* - y^*\rangle + \frac{\beta}{2}\|\M x^* - y^*\|^2 > -\infty,
  \end{equation}
  where the last inequality follows from the properness assumption on $P$.
  On the other hand, putting $M = 1$ and $N = t_i$ in \eqref{eq:bigrelation}, we see
  that
  \begin{equation}\label{eq:bigrelation2}
  L_\beta(x^{t_i},y^{t_i},z^{t_i}) - L_\beta(x^1,y^1,z^1)\le -\frac12\sum_{t=1}^{t_i-2}\|x^{t+1}-x^t\|_{\cR}^2 +\frac12\|x^{1} - x^{0}\|_{\frac{2}{\sigma\beta(1-\gamma)}\T_1^2}^2.
  \end{equation}
  Passing to the limit in \eqref{eq:bigrelation2} and making use of \eqref{eq:limitL} and (ii) in Assumption~\ref{assumption}, we conclude that
  \[
  0\ge -\frac12\sum_{t=1}^{\infty}\|x^{t+1}-x^t\|_{\cR}^2 > -\infty
  \]
  The desired relation \eqref{eq2:limit2} now follows from this and the fact that $\cR\succ 0$. Consequently, \eqref{eq2:limit} holds.

  We next show that \eqref{eq:Plimit} holds along the convergent subsequence in \eqref{subcon}. Indeed, from the definition of $y^{t_i}$ (as a minimizer), we have
  \[
  L_\beta(x^{t_i},y^{t_i+1},z^{t_i}) \le L_\beta(x^{t_i},y^*,z^{t_i}).
  \]
  Taking limit and using \eqref{subcon}, we see that
  \[
  \limsup_{i\rightarrow \infty}L_\beta(x^{t_i},y^{t_i+1},z^{t_i})\le h(x^*) + P(y^*) - \langle z^*,\M x^* - y^*\rangle + \frac{\beta}{2}\|\M x^* - y^*\|^2.
  \]
  On the other hand, from lower semicontinuity, \eqref{subcon} and \eqref{eq2:limit}, we have
  \[
  \liminf_{i\rightarrow \infty}L_\beta(x^{t_i},y^{t_i+1},z^{t_i})
  \ge h(x^*) + P(y^*) - \langle z^*,\M x^* - y^*\rangle + \frac{\beta}{2}\|\M x^* - y^*\|^2.
  \]
  The above two relations show that $\lim\limits_{i\rightarrow \infty} P(y^{t_i+1}) = P(y^*)$. This together with \eqref{eq2:limit} and the discussions preceding this theorem shows that $x^*$ is a stationary point of \eqref{P1} and that \eqref{eq:optcon} holds. This proves (i).

  Next, we suppose that the algorithm is initialized at
  a non-stationary $x^0$ with $h(x^0) + P(\M x^0) < \infty$ and $z^0$ chosen with $\M^* z^0 = \nabla h(x^0)$; we also write $y^0 = \M x^0$. We first show that $x^1\neq x^0$. To this end, we notice that
  \begin{equation*}
  \begin{split}
    \M^*(z^1 - z^0) &= \nabla h(x^1) + \nabla \phi(x^1) - \nabla \phi(x^0) - \M^* z^0\\
    &= \nabla h(x^1) - \nabla h(x^0) + \nabla \phi(x^1) - \nabla \phi(x^0).
  \end{split}
  \end{equation*}
  Proceeding as in \eqref{eq6}, we have
  \begin{equation}\label{eq8}
    \sigma\|z^1 - z^0\|^2 \le \frac1\gamma\|x^1-x^0\|_{\Q_3}^2.
  \end{equation}
  On the other hand, combining the relations $z^1 = z^0 - \beta(\M x^1 - y^1)$ and $y^0 = \M x^0$, we see that
  \begin{equation}\label{eq9}
    y^1 - y^0 = \M (x^1 - x^0) + \frac{1}{\beta}(z^1 - z^0).
  \end{equation}
  Consequently, if $x^1 = x^0$, then it follows from \eqref{eq8} and \eqref{eq9} that $z^1=z^0$ and $y^1=y^0$. This together with \eqref{eq:iterop} implies that
  \[
  0 \in \nabla h(x^0) + \M^* \partial P(\M x^0),
  \]
  i.e., $x^0$ is a stationary point.
  Since $x^0$ is non-stationary by assumption, we must have $x^1 \neq x^0$.

  We now derive an upper bound on $L_\beta(x^N,y^N,z^N) - L_\beta(x^0,y^0,z^0)$ for any $N > 1$. To this end,
  using the definition of augmented Lagrangian function, the $z$-update and \eqref{eq8}, we have
  \begin{equation*}
  \begin{split}
    L_\beta(x^1,y^1,z^1) - L_\beta(x^1,y^1,z^0) = \frac{1}{\beta}\|z^1-z^0\|^2 \le \frac{1}{\sigma\beta\gamma}\|x^1-x^0\|_{\Q_3}^2 \ .
  \end{split}
  \end{equation*}
  Combining this relation with \eqref{term2} and \eqref{term3}, we obtain
  the following estimate
  \begin{equation}\label{diffL1}
    L_\beta(x^1,y^1,z^1) - L_\beta(x^0,y^0,z^0) \le \frac12\|x^1-x^0\|_{\frac{2}{\sigma\beta\gamma}\Q_3-\delta \I -\T_2}^2.
  \end{equation}
  On the other hand, by specializing \eqref{eq:bigrelation}
  to $N > M = 1$ and recalling that $\cR \succ 0$, we see that
  \begin{equation}\label{diffLM}
    \begin{split}
     L_\beta(x^N,y^N,z^N) - L_\beta(x^1,y^1,z^1)&\le -\frac12\sum_{t=1}^{N-2}\|x^{t+1}-x^t\|_{\cR}^2 +\frac12\|x^1 - x^0\|_{\frac{2}{\sigma\beta(1-\gamma)}\T_1^2}^2\\
     & \le \frac12\|x^1 - x^0\|_{\frac{2}{\sigma\beta(1-\gamma)}\T_1^2}^2.
  \end{split}
  \end{equation}
  Combining \eqref{diffL1}, \eqref{diffLM} and the definition of $\cR$, we obtain
  \[
  L_\beta(x^N,y^N,z^N) - L_\beta(x^0,y^0,z^0) \le -\frac12\|x^1-x^0\|_{\cR}^2 < 0,
  \]
  where the strictly inequality follows from the fact that $x^1\neq x^0$, and the fact that $\cR \succ 0$. The conclusion of the theorem now follows by taking limit in the above inequality along any convergent subsequence, and noting that  $y^0 = \M x^0$ by assumption, and that $y^* = \M x^*$.
%
\end{proof}

We illustrate in the following examples how the parameters can be chosen in special cases.
\begin{example}\label{example:3}
  Suppose that $\M = \I$ and that $\nabla h$ is Lipschitz continuous with modulus bounded by $L$. Then one can take $\Q_1 = L \I$ and $\Q_2 = -L\I$. Moreover,
  Assumption~\ref{assumption}(i) holds with $\sigma = 1$. Furthermore, one can take $\phi(x) = \frac{L}{2}\|x\|^2 - h(x)$ so that $\T_1 = 2L\I$, $\T_2 = 0$ and $\Q_3 = L^2\I$. For the second and third points of Assumption~\ref{assumption}(ii) to hold, one can choose $\gamma = \frac12$ and then $\beta$ can be chosen so that $\beta - L = \delta > 0$
  and that
  \[
  \delta > \frac{4}{\beta}L^2 + \frac{4}{\beta}(2L)^2 = \frac{20}{\beta}L^2.
  \]
  These can be achieved by picking $\beta > 5L$.
\end{example}

\begin{example}\label{example:4}
  Suppose again that $\M = \I$ and $h(x) = \frac{1}{2}\|\A x - b\|^2$ for some linear map $\A$ and vector $b$. Then one can take $\phi = 0$
  so that $\T_1 = \T_2 = 0$, and $\Q_1 = L\I$, $\Q_2 = 0$, $\Q_3 = L^2\I$, where $L = \lambda_{\max}(\A^*\A)$. Observe that Assumption~\ref{assumption}(i) holds with $\sigma = 1$.
  For the second and third points of Assumption~\ref{assumption}(ii) to hold, we only need to pick $\beta$ so that $\beta = \delta > \frac{2}{\beta}L^2$, i.e., $\beta > \sqrt{2}L$,
  while $\gamma$ can be any number chosen from $(\frac{\sqrt{2}L}{\beta},1)$.
\end{example}

\begin{example}\label{example:5}
  Suppose that $\M$ is a general surjective linear map and $h$ is strongly convex. Specifically, assume that $h(x) = \frac{1}{2}\|x - \hat x\|^2$ for some $\hat x$ so that $\Q_1=\Q_2= \I$. Then we can take $\phi=0$ and hence $\T_1=\T_2=0$, $\Q_3 = \I$. Assumption~\ref{assumption}(i) holds with $\sigma = \lambda_{\min}(\M\M^*)$.
  The second point of Assumption~\ref{assumption}(ii) holds with $\delta = 1$. For the third point to hold, it suffices to pick $\beta > 2/\sigma$, while $\gamma$ can be any number chosen from $(\frac{2}{\sigma\beta},1)$.
\end{example}


We next give some sufficient conditions under which the sequence $\{(x^t,y^t,z^t)\}$ generated from the proximal ADMM under Assumption~\ref{assumption} is bounded. This would guarantee the existence of cluster point, which is the assumption required in Theorem~\ref{thm:main}.

\begin{theorem}{\bf (Boundedness of sequence generated from the proximal ADMM)}\label{thm:boundedness2}
  Suppose that Assumption~\ref{assumption} holds, and $\beta$ is further chosen so that there exists $0<\zeta < 2\beta\gamma$ with
  \begin{equation}\label{hlowerbound2}
    \inf_{x} \left\{h(x) - \frac{1}{\sigma\zeta}\|\nabla h(x)\|^2\right\} =:h_0 > -\infty.
  \end{equation}
  Suppose that either
  \begin{enumerate}[{\rm (i)}]
    \item $\M$ is invertible and $\liminf_{\|y\|\rightarrow \infty}P(y)=\infty$; or
    \item $\liminf_{\|x\|\rightarrow \infty} h(x)=\infty$ and $\inf_y P(y) > -\infty$.
  \end{enumerate}
  Then the sequence $\{(x^t,y^t,z^t)\}$ generated from the proximal ADMM is bounded.
\end{theorem}
\begin{proof}
  First, observe from \eqref{eq:smallrelation} that
  \begin{equation*}
  \begin{split}
  \left(L_\beta(x^{t+1},y^{t+1},z^{t+1}) +\frac12\|x^{t+1} - x^t\|_{\frac{2}{\sigma\beta(1-\gamma)}\T_1^2}^2\right)& - \left(L_\beta(x^t,y^t,z^t) +\frac12\|x^{t} - x^{t-1}\|_{\frac{2}{\sigma\beta(1- \gamma)}\T_1^2}^2\right) \\
  &\le\frac12\|x^{t+1}-x^t\|_{\frac{2}{\sigma\beta}\H_\gamma-\delta \I -\T_2}^2\le 0,
  \end{split}
  \end{equation*}
  where the last inequality follows from point 3 in (ii) of Assumption~\ref{assumption}. In particular, the sequence $\{L_\beta(x^t,y^t,z^t) +\frac12\|x^{t} - x^{t-1}\|_{\frac{2}{\sigma\beta(1- \gamma)}\T_1^2}^2\}$ is decreasing and consequently, we have, for all $t\ge 1$, that
  \begin{equation}\label{eq:smallrelation4}
    L_\beta(x^t,y^t,z^t) +\frac12\|x^{t} - x^{t-1}\|_{\frac{2}{\sigma\beta(1- \gamma)}\T_1^2}^2\le L_\beta(x^1,y^1,z^1) +\frac12\|x^1 - x^0\|_{\frac{2}{\sigma\beta(1- \gamma)}\T_1^2}^2.
  \end{equation}
  Next, recall from \eqref{Mzhrelation} that
  \begin{equation}\label{zbound2}
  \begin{split}
    \sigma\|z^t\|^2 & \le \|\M^* z^t\|^2 = \|\nabla h(x^t) + \nabla \phi(x^t) - \nabla \phi(x^{t-1})\|^2\\
    & \le \frac1\gamma\|\nabla h(x^t)\|^2 + \frac1{1- \gamma}\|\nabla \phi(x^t) - \nabla \phi(x^{t-1})\|^2\\
    & \le \frac1\gamma\|\nabla h(x^t)\|^2 + \frac1{1- \gamma}\|x^t - x^{t-1}\|_{\T_1^2}^2.
  \end{split}
  \end{equation}
  Plugging this into \eqref{eq:smallrelation4}, we see further that
  \begin{equation}\label{last_rel2}
    \begin{split}
      & L_\beta(x^1,y^1,z^1) +\frac12\|x^1 - x^0\|_{\frac{2}{\sigma\beta(1- \gamma)}\T_1^2}^2 \ge L_\beta(x^t,y^t,z^t) +\frac12\|x^{t} - x^{t-1}\|_{\frac{2}{\sigma\beta(1- \gamma)}\T_1^2}^2\\
      & = h(x^t) + P(y^t) + \frac\beta2 \left\|\M x^t - y^t - \frac{z^t}{\beta}\right\|^2 - \frac{1}{2\beta}\|z^t\|^2  +\frac12\|x^{t} - x^{t-1}\|_{\frac{2}{\sigma\beta(1- \gamma)}\T_1^2}^2\\
      & \ge h(x^t) + P(y^t) + \frac\beta2 \left\|\M x^t - y^t - \frac{z^t}{\beta}\right\|^2 - \frac{1}{2\sigma\beta\gamma}\|\nabla h(x^t)\|^2 + \frac12\|x^{t} - x^{t-1}\|_{\frac{1}{\sigma\beta(1-\gamma)}\T_1^2}^2 \\
      & = \mu h(x^t) + (1-\mu)h(x^t) + P(y^t) + \frac\beta2 \left\|\M x^t - y^t - \frac{z^t}{\beta}\right\|^2 - \frac{1}{2\sigma\beta\gamma}\|\nabla h(x^t)\|^2  +\frac12\|x^{t} - x^{t-1}\|_{\frac{1}{\sigma\beta(1- \gamma)}\T_1^2}^2 \\
      & \ge \mu h(x^t) + (1-\mu)h_0 + \frac{c}\sigma \|\nabla h(x^t)\|^2 + P(y^t) + \frac\beta2 \left\|\M x^t - y^t - \frac{z^t}{\beta}\right\|^2
      +\frac12\|x^{t} - x^{t-1}\|_{\frac{1}{\sigma\beta(1-\gamma)}\T_1^2}^2,
    \end{split}
  \end{equation}
  where $c:= \frac{1-\mu}{\zeta} - \frac1{2\beta\gamma}$, and $\mu \in (0,1)$ is chosen so that $(1-\mu)\beta > \zeta/(2\gamma)$, i.e., $c > 0$.

  Now, suppose that the conditions in {\rm (i)} hold. Note that $\liminf_{\|y\|\rightarrow \infty}P(y)=\infty$ implies $\inf_y P(y) > -\infty$. This together with \eqref{last_rel2} and $(1-\mu)\beta > \zeta/(2\gamma)$ implies that $\{y^t\}$, $\{\nabla h(x^t)\}$, and $\{\|x^t - x^{t-1}\|_{\T_1^2}\}$ are bounded. Boundedness of $\{z^t\}$ follows from these and \eqref{zbound2}. Moreover, the boundedness of $\{x^t\}$ follows from the boundedness of $\{y^t\}$, $\{z^t\}$, the invertibility of $\M$ and the third relation in \eqref{scheme}.
  Next, consider the conditions in {\rm (ii)}. Since $P$ is bounded below, \eqref{last_rel2} and the coerciveness of $h(x)$ give the boundedness of $\{x^t\}$. The boundedness of $\{z^t\}$ follows from this and \eqref{zbound2}. Finally, the boundedness of $\{y^t\}$ follows from these and the third relation in \eqref{scheme}. This completes the proof.
\end{proof}

Notice that in order to guarantee boundedness of the sequence generated from the proximal ADMM, we have to choose $\beta$ to satisfy both Assumption~\ref{assumption} and \eqref{hlowerbound2}. We illustrate the conditions in Theorem~\ref{thm:boundedness2} in the next few examples. In particular, we shall see that such a choice of $\beta$ does exist in the following examples.

\begin{example}\label{examplenew:3}
  Consider the problem in Example~\ref{example:3}, and suppose in addition that $h(x) = \frac12\|\A x - b\|^2$ for some linear map $\A$ and vector $b$, and that $P$ is coercive, i.e., $\liminf_{\|y\|\rightarrow \infty}P(y)=\infty$. This includes the model of $\ell_\frac12$ regularization considered in \cite{ZengLinWangXu14}. Since $h(x) = \frac12 \|\A x - b\|^2$, we have
  \begin{equation}\label{ex:hbound}
  h(x) - \frac{1}{2\sqrt{2}L}\|\nabla h(x)\|^2 = \frac12 \|\A x - b\|^2 - \frac{1}{2\sqrt{2}L}\|\A^*(\A x - b)\|^2\ge \frac12\left(1 - \frac{1}{\sqrt{2}}\right)\|\A x - b\|^2 \ge 0.
  \end{equation}
  where $L = \lambda_{\max}(\A^*\A)$.
  Thus, \eqref{hlowerbound2} holds with $\sigma = 1$ and $\zeta = 2\sqrt{2}L < 5L < 2\beta\gamma$, where $\gamma = \frac12$. Hence, the sequence generated from the proximal ADMM is bounded, according to Theorem~\ref{thm:boundedness2} {\rm (i)}.
\end{example}

\begin{example}\label{examplenew:1}
  Consider the problem in Example~\ref{example:4}, and suppose in addition that $P$ is coercive, i.e., $\liminf_{\|y\|\rightarrow \infty}P(y)=\infty$. This covers the model of $\ell_\frac12$ regularization considered in \cite{ZengLinWangXu14}. We show that $\{(x^t,y^t,z^t)\}$ is bounded by verifying the conditions in Theorem~\ref{thm:boundedness2}.
  Indeed, we have from \eqref{ex:hbound} that \eqref{hlowerbound2} holds with $\sigma = 1$ and $\zeta = 2\sqrt{2}L < 2\beta\gamma$; recall that $L = \lambda_{\max}(\A^*\A)$ and $\gamma$ can be chosen from $(\frac{\sqrt{2}L}{\beta},1)$ in this example. The conclusion now follows from Theorem~\ref{thm:boundedness2} {\rm (i)}.
\end{example}

\begin{example}\label{examplenew:2}
  Consider the problem in Example~\ref{example:5}, and assume in addition that $\inf_y P(y) > -\infty$. We show that $\{(x^t,y^t,z^t)\}$ is bounded by showing that \eqref{hlowerbound2} holds for our choice of $\beta$. The conclusion will then follow from Theorem~\ref{thm:boundedness2} {\rm (ii)}.

  To this end, note that $h(x) = \frac12 \|x - \hat x\|^2$ and thus
  \[
  h(x) - \frac{1}{4}\|\nabla h(x)\|^2 = \frac14 \|x - \hat x\|^2 \ge 0.
  \]
  Thus, \eqref{hlowerbound2} holds with $\zeta = 4/\sigma < 2\beta\gamma$; recall that $\gamma$ can be chosen from $(\frac{2}{\sigma\beta},1)$ in this example.
\end{example}

\begin{remark}
  We further comment on the condition \eqref{hlowerbound2}. In particular, we shall argue that for a fairly large class of twice continuously differentiable function $h$ with a bounded Hessian, there exists $\nu > 0$ so that
  \[
  \inf_x\left\{h(x) - \frac{1}{2\nu}\|\nabla h(x)\|^2\right\} > -\infty.
  \]
  Actually, let $h$ be a twice continuously differentiable function with a bounded Hessian and $\inf\limits_{x}h(x) =:\alpha > -\infty$. Then it is well known that
  \[
  \inf_{x}\left\{h(x) - \frac{1}{2L}\|\nabla h(x)\|^2\right\} > -\infty,
  \]
  where $L$ is a Lipschitz continuity modulus of $\nabla h(x)$. We include a simple proof for the convenience of the readers. Indeed,
  \begin{equation*}
    \begin{split}
       \alpha \le h\left(x - \frac1L \nabla h(x)\right) &\le h(x) + \left\langle\nabla h(x), \left(x - \frac1L \nabla h(x)\right) - x\right\rangle + \frac{L}2\left\|\left(x - \frac1L \nabla h(x)\right) - x\right\|^2\\
      & = h(x) - \frac{1}{2L}\|\nabla h(x)\|^2,
    \end{split}
  \end{equation*}
  where the first inequality follows from the fact that $h$ is bounded from below by $\alpha$, and the second inequality follows from the fact that the gradient is Lipschitz continuous with modulus $L$. Consequently, for a twice continuously differentiable function $h$ with a bounded Hessian, the condition \eqref{hlowerbound2} holds for some $\sigma\zeta > 0$ if and only if $h$ is bounded below.
\end{remark}

We now study convergence of the whole sequence generated by the ADMM (i.e., proximal ADMM with $\phi = 0$)
when the objective function is semi-algebraic. The proof of this theorem
relies heavily on the KL property. For recent applications of KL property to convergence analysis of a broad class of optimization methods, see \cite{AtBoSv13}. We would like to point out that our analysis is adapted from \cite{AtBoSv13}, and we cannot directly apply the results there since some of their assumptions are not satisfied in our settings. We will further comment on this in Remark~\ref{rem4}.

\begin{theorem}{\bf (Global convergence for the whole sequence)}\label{th:2}
  Suppose that Assumption~\ref{assumption} holds with $\T_1 = 0$ (and hence $\phi = 0$), and that $h$ and $P$ are semi-algebraic functions. Suppose further that the sequence $\{(x^t,y^t,z^t)\}$ generated from the ADMM has a cluster point $(x^*,y^*,z^*)$. Then the sequence $\{(x^t,y^t,z^t)\}$ converges to $(x^*,y^*,z^*)$ and $x^*$ is a stationary point of \eqref{P1}. Moreover,
  \begin{equation}\label{eq:sumx}
    \sum_{t=1}^\infty\|x^{t+1} - x^t\| < \infty.
  \end{equation}
\end{theorem}

\begin{proof}
  The conclusion that $x^*$ is a stationary point of \eqref{P1} follows from Theorem~\ref{thm:main}. Moreover, \eqref{eq2:limit} holds.
  We now establish convergence.

  First, consider the subdifferential of $L_\beta$ at $(x^{t+1},y^{t+1},z^{t+1})$. Specifically, we have
  \begin{equation*}
  \begin{split}
    \nabla_x L_\beta(x^{t+1},y^{t+1},z^{t+1}) & = \nabla h(x^{t+1}) - \M^* z^{t+1} + \beta\M^*(\M x^{t+1} - y^{t+1}) \\
    & = \beta\M^*(\M x^{t+1} - y^{t+1}) = -\M^*(z^{t+1} - z^t),
  \end{split}
  \end{equation*}
  where the last two equalities follow from the second and third relations in \eqref{eq:iterop}. Similarly,
  \begin{equation*}
    \begin{split}
    \nabla_z L_\beta(x^{t+1},y^{t+1},z^{t+1}) &= - (\M x^{t+1} - y^{t+1}) = \frac{1}{\beta}(z^{t+1} - z^t).\\
    \partial_y L_\beta(x^{t+1},y^{t+1},z^{t+1}) & = \partial P(y^{t+1}) + z^{t+1} - \beta(\M x^{t+1} - y^{t+1})\\
    & \ni z^{t+1} - z^t - \beta \M (x^{t+1} - x^t),
  \end{split}
  \end{equation*}
  since $0\in \partial P(y^{t+1}) + z^t - \beta(\M x^t - y^{t+1})$ from \eqref{eq1}. The above relations together with the assumption that $\T_1=0$
  and \eqref{eq6} imply the existence of a constant $C > 0$ so that
  \begin{equation}\label{eq:disL}
    {\rm dist}(0,\partial L_\beta(x^{t+1},y^{t+1},z^{t+1})) \le C \|x^{t+1} - x^t\|.
  \end{equation}
  Moreover, from \eqref{eq:smallrelation} and $\T_1 = 0$ (and hence $\T_2 = 0$), we see that
  \begin{equation}\label{eq:smallrelation2}
  L_\beta(x^t,y^t,z^t) - L_\beta(x^{t+1},y^{t+1},z^{t+1})\ge -\frac12\|x^{t+1}-x^t\|_{\frac{2}{\sigma\beta\gamma}\Q_3-\delta \I}^2 \ge D \|x^{t+1}-x^t\|^2
  \end{equation}
  for some $D > 0$. In particular, $\{L_\beta(x^t,y^t,z^t)\}$ is decreasing.
  Since  $L_\beta$ is also bounded below along the subsequence in \eqref{subcon},
  we conclude that $\displaystyle \lim_{t \rightarrow \infty} L_\beta(x^t,y^t,z^t)$ exists.

  We now show that $\lim\limits_{t \rightarrow \infty} L_\beta(x^t,y^t,z^t) = l^*$; here, we write $l^*:= L_\beta(x^*,y^*,z^*)$ for notational simplicity. To this end, notice from the definition of $y^{t+1}$ as a minimizer that
  \[
  L_\beta(x^t,y^{t+1},z^t)\le L_\beta(x^t,y^*,z^t).
  \]
  Using this relation, \eqref{eq2:limit} and the continuity of $L_\beta$ with respect to the $x$ and $z$ variables, we have
  \begin{equation}\label{eq:upper}
    \limsup_{j\rightarrow \infty}L_\beta(x^{t_j+1},y^{t_j+1},z^{t_j+1}) \le L_\beta(x^*,y^*,z^*),
  \end{equation}
  where $\{(x^{t_j},y^{t_j},z^{t_j})\}$ is a subsequence that converges to $(x^*,y^*,z^*)$. On the other hand, from \eqref{eq2:limit}, we see that $\{(x^{t_j+1},y^{t_j+1},z^{t_j+1})\}$ also converges to $(x^*,y^*,z^*)$. This together with the lower semicontinuity of $L_\beta$ imply
  \begin{equation}\label{eq:lower}
    \liminf_{j\rightarrow \infty}L_\beta(x^{t_j+1},y^{t_j+1},z^{t_j+1}) \ge L_\beta(x^*,y^*,z^*).
  \end{equation}
  Combining \eqref{eq:upper}, \eqref{eq:lower} and the existence of $\lim L_\beta(x^t,y^t,z^t)$, we conclude that
  \begin{equation}\label{eq:limitL2}
    \lim_{t\rightarrow \infty} L_\beta(x^t,y^t,z^t) = l^*,
  \end{equation}
  as claimed. Furthermore, if $L_\beta(x^t,y^t,z^t) = l^*$ for some $t\ge 1$, since the sequence is decreasing, we must have $L_\beta(x^t,y^t,z^t) = L_\beta(x^{t+k},y^{t+k},z^{t+k})$ for all $k\ge 0$. From \eqref{eq:smallrelation2}, we see that $x^t = x^{t+k}$ and hence $z^t = z^{t+k}$ from the fact that $\T_1 = 0$ and \eqref{eq6}, for all $k\ge 0$. Consequently, we conclude from \eqref{eq7} that $y^{t+1} = y^{t+k}$ for all $k\ge 1$, meaning that the algorithm terminates finitely. Since the conclusion of this theorem holds trivially if the algorithm terminates finitely, from now on, we only consider the case where $L_\beta(x^t,y^t,z^t) > l^*$ for all $t\ge 1$.

  Next, notice that the function $(x,y,z) \mapsto L_\beta(x,y,z)$ is semi-algebraic
  due to the semi-algebraicity of $h$ and $P$. Thus, it is a KL function from \cite[Section~4.3]{AtBoReSo10}. From the property of KL functions, there exist $\eta > 0$, a neighborhood $V$ of $(x^*,y^*,z^*)$ and a continuous concave function $\varphi:[0,\eta)\rightarrow {\mathbb{R}}_+$ as described in Definition~\ref{def:KL} so that
  for all $(x,y,z)\in V$ satisfying $l^*< L_\beta(x,y,z) < l^* + \eta$, we have
    \begin{equation}\label{eq:KL_ineq2}
      \varphi'(L_\beta(x,y,z) - l^*)\,{\rm dist}(0,\partial L_\beta(x,y,z))\ge 1.
    \end{equation}

  Pick $\rho > 0$ so that
  \[
  {\bf B}_\rho:=\left\{(x,y,z):\; \|x - x^*\|<\rho, \|y - y^*\| < (\|\M\| + 1)\rho, \|z-z^*\| < \sqrt{\frac{\lambda_{\max}(\Q_3)}{\sigma}}\rho\right\}\subseteq V
  \]
  and set $B_\rho:= \{x:\; \|x - x^*\| < \rho\}$.
  From the second relation in \eqref{eq:iterop} and \eqref{eq:optcon}, we obtain for any $t\ge 1$ that
  \begin{equation*}
    \sigma\|z^t - z^*\|^2 \le \|\M^*(z^t - z^*)\|^2 = \|\nabla h(x^t) - \nabla h(x^*)\|^2 \le \lambda_{\max}(\Q_3)\|x^t - x^*\|^2.
  \end{equation*}
  Hence $\|z^t - z^*\|< \sqrt{\frac{\lambda_{\max}(\Q_3)}{\sigma}}\rho$ whenever $x^t\in B_\rho$ and $t\ge 1$.
  Moreover, from the definition of $z^{t+1}$ and \eqref{eq:optcon}, we see that whenever $t\ge 1$,
  \begin{equation*}
    \|y^t - y^*\| = \left\| \M (x^t - x^*) + \frac{1}{\beta}(z^t - z^{t-1})\right\| \le \|\M\|\|x^t - x^*\| + \frac{1}{\beta}\|z^t - z^{t-1}\|.
  \end{equation*}
  Since there exists $N_0\ge 1$ so that for all $t\ge N_0$, we have $\|z^t - z^{t-1}\| < \beta\rho$ (such an $N_0$ exists due to \eqref{eq2:limit}), it follows that
  $\|y^t - y^*\| < (\|\M\| + 1)\rho$ whenever $x^t\in B_\rho$ and $t\ge N_0$. Thus, if $x^t\in B_\rho$ and $t\ge N_0$, we have $(x^t,y^t,z^t)\in {\bf B}_\rho\subseteq V$.
  Moreover, it is not hard to see that there exists $(x^N,y^N,z^N)$ with $N \ge N_0$ such that
  \begin{enumerate}[{\rm (i)}]
    \item $x^N\in B_\rho$;
    \item $l^*< L_\beta(x^N,y^N,z^N) < l^* + \eta$;
    \item $
  \|x^N - x^*\| + 2 \sqrt{\frac{L_\beta(x^N,y^N,z^N) - l^*}{D}} + \frac{C}{D}\varphi(L_\beta(x^N,y^N,z^N) - l^*) < \rho$.
  \end{enumerate}
  Indeed, these properties follow from the fact that $(x^*,y^*,z^*)$ is a cluster point, \eqref{eq:limitL2} and that $L_\beta(x^t,y^t,z^t)>l^*$ for all $t\ge 1$.

  We next show that, if $x^t\in B_\rho$ and $l^* < L_\beta(x^t,y^t,z^t) < l^* + \eta$ for some fixed $t\ge N_0$, then
  \begin{equation}\label{eq:toprove}
  \begin{split}
    &\|x^{t+1} - x^t\| + (\|x^{t+1} - x^t\| - \|x^{t} - x^{t-1}\|)\\
    &\le \frac{C}{D}[\varphi(L_\beta(x^t,y^t,z^t) - l^*) - \varphi(L_\beta(x^{t+1},y^{t+1},z^{t+1}) - l^*)].
  \end{split}
  \end{equation}
  To see this, notice that $x^t\in B_\rho$ and $t\ge N_0$ implies $(x^t,y^t,z^t)\in {\bf B}_\rho\subseteq V$. Hence, \eqref{eq:KL_ineq2} holds for $(x^t,y^t,z^t)$. Combining \eqref{eq:disL}, \eqref{eq:smallrelation2}, \eqref{eq:KL_ineq2} and the concavity of $\phi$, we conclude that for all such $t$
  \begin{equation*}
    \begin{split}
      & C \|x^{t} - x^{t-1}\|\cdot[\varphi(L_\beta(x^t,y^t,z^t) - l^*) - \varphi(L_\beta(x^{t+1},y^{t+1},z^{t+1}) - l^*)]\\
      &\ge  {\rm dist}(0,\partial L_\beta(x^t,y^t,z^t))\cdot[\varphi(L_\beta(x^t,y^t,z^t) - l^*) - \varphi(L_\beta(x^{t+1},y^{t+1},z^{t+1}) - l^*)]\\
      &\ge  {\rm dist}(0,\partial L_\beta(x^t,y^t,z^t))\cdot\varphi'(L_\beta(x^t,y^t,z^t) - l^*)\cdot[L_\beta(x^t,y^t,z^t) - L_\beta(x^{t+1},y^{t+1},z^{t+1})]\\
      &\ge  D \|x^{t+1}-x^t\|^2.
    \end{split}
  \end{equation*}
  Dividing both sides by $D$, taking square root, using the inequality $2\sqrt{ab}\le a + b$ as in the proof of \cite[Lemma~2.6]{AtBoSv13}, and rearranging terms, we conclude
  that \eqref{eq:toprove} holds.

  We now show that $x^t\in B_\rho$ whenever $t\ge N$.
  We establish this claim by induction, and our proof is similar to the proof of \cite[Lemma~2.6]{AtBoSv13}.
  The claim is true for $t=N$ by construction. For $t= N+1$, we have
  \begin{equation*}
    \begin{split}
      &\|x^{N+1} - x^*\| \le \|x^{N+1}-x^N\| + \|x^N - x^*\| \\
      &\le \sqrt{\frac{L_\beta(x^N,y^N,z^N) - L_\beta(x^{N+1},y^{N+1},z^{N+1})}{D}} + \|x^N - x^*\|\\
      &\le \sqrt{\frac{L_\beta(x^N,y^N,z^N) - l^*}{D}} + \|x^N - x^*\| < \rho,
    \end{split}
  \end{equation*}
  where the first inequality follows from \eqref{eq:smallrelation2}.
  Now, suppose the claim is true for $t=N,\ldots,N+k-1$ for some $k> 1$; i.e., $x^N,\ldots,x^{N+k-1}\in B_\rho$. We now consider the case when $t = N+k$:
  \begin{equation*}
    \begin{split}
      &\|x^{N+k} - x^*\| \le \|x^N - x^*\| + \|x^N - x^{N+1}\| + \sum_{j=1}^{k-1}\|x^{N+j+1}-x^{N+j}\|\\
      & = \|x^N - x^*\| + 2\|x^N - x^{N+1}\| - \|x^{N+k}-x^{N+k-1}\|\\
      & \ \ +\sum_{j=1}^{k-1}[\|x^{N+j+1}-x^{N+j}\| + (\|x^{N+j+1}-x^{N+j}\| - \|x^{N+j}-x^{N+j-1}\|)]\\
      & \le \|x^N - x^*\| + 2\|x^N - x^{N+1}\|\\
      & \ \ + \frac{C}{D}\sum_{j=1}^{k-1}[\varphi(L_\beta(x^{N+j},y^{N+j},z^{N+j}) - l^*) - \varphi(L_\beta(x^{N+j+1},y^{N+j+1},z^{N+j+1}) - l^*)]\\
      & \le \|x^N - x^*\| + 2\|x^N - x^{N+1}\| + \frac{C}{D}\varphi(L_\beta(x^{N+1},y^{N+1},z^{N+1}) - l^*),
    \end{split}
  \end{equation*}
  where the first inequality follows from \eqref{eq:toprove}, the monotonicity of $\{L_\beta(x^t,y^t,z^t)\}$ from \eqref{eq:smallrelation2}, and the induction assumption that $x^N,\ldots,x^{N+k-1}\in B_\rho$. Moreover, in view of \eqref{eq:smallrelation2} and the definition of $\rho$, we see that the last expression above is less than $\rho$. Hence, $\|x^{N+k} - x^*\|<\rho$ as claimed, and we have shown that $x^t\in B_\rho$ for $t\ge N$ by induction.

  Since $x^t\in B_\rho$ for $t\ge N$, we can sum \eqref{eq:toprove} from $t=N$ to $M\rightarrow \infty$. Invoking \eqref{eq2:limit}, we arrive at
  \begin{equation*}
    \sum_{t=N}^\infty\|x^{t+1} - x^t\| \le  \frac{C}{D}\varphi(L_\beta(x^N,y^N,z^N) - l^*) + \|x^N - x^{N-1}\|,
  \end{equation*}
  which implies that \eqref{eq:sumx} holds. Convergence of $\{x^t\}$ follows immediately from this. Convergence of $\{y^t\}$ follows from the convergence of $\{x^t\}$, the relation $y^{t+1} = \M x^{t+1} + \frac{1}{\beta}(z^{t+1} - z^t)$ from \eqref{scheme}, and \eqref{eq2:limit}. Finally, the convergence of $\{z^t\}$ follows from the surjectivity of $\M$, and the relation $\M^*z^{t+1} = \nabla h(x^{t+1})$ from \eqref{eq:iterop}. This completes the proof.
\end{proof}
\begin{remark}{\bf (Comments on Theorem \ref{th:2})}\label{rem4}
\begin{itemize}
 \item[{\rm (1)}] A close inspection of the above proof shows that the conclusion of Theorem \ref{th:2} continues to hold as long as the augmented Lagrangian $L_\beta$ is a KL-function.
Here, we only state the case where $h$ and $P$ are semi-algebraic because this simple sufficient condition can be easily verified.
\item[{\rm (2)}]  Although a general convergence analysis framework was established in \cite{AtBoSv13} for a broad class of optimization problems, it is not clear to us whether their results can be applied directly here. Indeed, to ensure convergence, three basic properties {\bf H1}, {\bf H2} and {\bf H3} were imposed in \cite[Page~99]{AtBoSv13}.
In particular, their property {\bf H1} (sufficient descent property) in our case reads: $$L_\beta(x^t,y^t,z^t) - L_\beta(x^{t+1},y^{t+1},z^{t+1}) \ge D (\|x^{t+1}-x^t\|^2+\|y^{t+1}-y^t\|^2+\|z^{t+1}-z^t\|^2),
$$ for some $D>0$. On the other hand, (\ref{eq:smallrelation2}) in our proof only gives us that
$L_\beta(x^t,y^t,z^t) - L_\beta(x^{t+1},y^{t+1},z^{t+1}) \ge D \|x^{t+1}-x^t\|^2$, which is not sufficient
for property {\bf H1} to hold.

\item[{\rm (3)}] In Theorem \ref{th:2}, we only discussed the case where $\phi=0$. This condition is used to ensure that $\{L_\beta(x^t,y^t,z^t)\}$ is a decreasing
sequence that is at least as large as $L_\beta(x^*,y^*,z^*)$.  It would be interesting to see whether the analysis here can be further extended to the case where $\phi \neq 0$.
\end{itemize}

\end{remark}

Before ending this section, we comment on the behavior of ADMM \eqref{scheme} in the case where $\M$ is assumed to be injective (instead of surjective). As suggested by the numerical experiments in \cite{DongZhang13} and our preliminary numerical tests, it is conceivable that the ADMM does {\em not} cluster at a stationary point in general when applied to solving problem \eqref{P1} with an injective $\M$. We hereby give a concrete 2-dimensional example for non-convergence, motivated by the recent counterexample in \cite[Remark~6]{Bauschke} for the convergence of Douglas-Rachford splitting method in a nonconvex setting.\footnote{Douglas-Rachford (DR) splitting method is a popular method for nonconvex feasibility problems and can be suitably applied to solving \eqref{P1} when $\M = \I$; see \cite{LiPong14_2}. Moreover, it has been brought to our attention during the revision process of this paper that the known equivalence between the ADMM and the DR splitting method in the convex case (see, for example, \cite[Remark~3.14]{Bau_Koch}) can be passed through to the nonconvex cases. Thus, the global convergence results in this paper concerning the ADMM can be specialized to obtain global convergence of the DR splitting method in some nonconvex settings. We note that the global convergence of the DR splitting method in the nonconvex settings has been studied in \cite{LiPong14_2} based on a new specially constructed merit function.}

\begin{example}\label{ex7:nonconverge}{\bf (Divergence of ADMM \eqref{scheme} when $\M$ is injective)}
  Fix $\eta \in (0,1]$ and set $C = \{x\in \R^2:\; x_2 = 0\}$ and $D = \{(0,0),(2,\eta),(2,-\eta)\}$. Then $C\cap D \neq \emptyset$. Consider the optimization problem
  \begin{equation*}
    \begin{array}{rl}
      \min\limits_{x} & 0\\
      {\rm s.t.} & x\in C,\ x\in D.
    \end{array}
  \end{equation*}
  This problem corresponds to \eqref{P1} with $h(x) = 0$, $P(y) = \delta_C(y_1) + \delta_D(y_2)$ where $y = (y_1,y_2)$, and $\M$ is the linear map so that $\M x = (x,x)$; the problem can be equivalently reformulated as
  \begin{equation*}
    \begin{array}{rl}
      \min\limits_{x,y} & 0\\
      {\rm s.t.} & x - y_1 = 0,\\
      & x - y_2 = 0,\\
      & y_1\in C,\ y_2\in D,
    \end{array}
  \end{equation*}
  and the ADMM can be applied.
  Let $z_1$ and $z_2$ denote the multipliers corresponding to the first and second equality constraints, respectively. The iterates in \eqref{scheme} (with $\phi = 0$) now take the form
  \begin{equation}\label{counterscheme}
\left\{
\begin{split}
&y_1^{t+1} = P_C\left(x^t - \frac{z_1^t}{\beta}\right),\ y_2^{t+1} \in P_D\left(x^t - \frac{z_2^t}{\beta}\right), \\
&x^{t+1} = \frac{1}{2}\left(y_1^{t+1} + \frac{z_1^t}{\beta} + y_2^{t+1} + \frac{z_2^t}{\beta} \right),\\
&z_1^{t+1} = z_1^t - \beta  (x^{t+1}- y_1^{t+1}),\\
&z_2^{t+1} = z_2^t - \beta  (x^{t+1}- y_2^{t+1}).
\end{split}
\right.
\end{equation}
For concreteness, whenever ambiguity arises in updating $y_2^{t+1}$ via the projection onto the nonconvex (discrete) set $D$, we choose the element in $D$ that is closest to the previous iterate $y_2^t$.

For each $\beta > 0$, consider the initializations $x^0 = (2,0)$, $z_1^0 = (0,-\beta \eta)$ and $z_2^0 = (0,\beta\eta)$. Then it is routine to show that the ADMM described in \eqref{counterscheme} will exhibit a discrete limit cycle of length $8$. Specifically, $(y_1^t,y_2^t,x^t,z_1^t,z_2^t) = (y_1^{8k+t},y_2^{8k+t},x^{8k+t},z_1^{8k+t},z_2^{8k+t})$ for any $1\le t\le 8$ and $k \ge 0$. Moreover,
\begin{equation*}
  \begin{split}
    y_1^t = (2,0),\ \ &1\le t\le 8,\ \ \
    y_2^t = \begin{cases}
      (2,-\eta) & 1\le t\le 4,\\
      (2,\eta) & 5\le t\le 8,
    \end{cases}\ \ \
    x^t = \begin{cases}
      (2,-\frac\eta2) & 1\le t\le 4,\\
      (2,\frac\eta2) & 5\le t\le 8,
    \end{cases}\\
  &z_1^t = \left(0,\frac{(2 - |t - 4|)\beta\eta}2\right),\ \ 1\le t\le 8,\ \ \ z_2^t = - z_1^t.
  \end{split}
\end{equation*}
In particular, the sequence $\{x^t\}$ is not convergent and the successive change of the $z$-update does not converge to zero.
\end{example}

\section{Proximal gradient algorithm when $\M = \I$}\label{sec:MI}

In this section, we look at the model problem \eqref{P1} in the case where $\M=\I$. 
Since the objective is the sum of a smooth and a possibly nonsmooth part with a simple proximal mapping, it is natural to consider the proximal gradient algorithm (also known as the forward-backward splitting algorithm). In this approach, one considers the update
\begin{equation}\label{FBSupdate}
  x^{t+1} \in \Argmin_x\left\{\langle \nabla h(x^t),x - x^t\rangle + \frac{1}{2\beta}\|x - x^t\|^2 + P(x) \right\}.
\end{equation}
From our assumption on $P$, the update can be performed efficiently via a computation of the proximal mapping of $\beta P$.
When $\beta \in (0,\frac{1}{L})$, where $L\ge \sup\{\|\nabla^2h(x)\|:\; x\in \R^n\}$,
it is not hard to show that any cluster point $x^*$ of the sequence generated above is a stationary point
of \eqref{P1}; see, for example, \cite{BreLo09}. 
In what follows, we analyze the convergence under a slightly more flexible step-size rule.
\begin{theorem}\label{prop:prox}
  Suppose that there exists a twice continuously differentiable convex function $q$ and $\ell > 0$ such that for all $x$,
  \begin{equation}\label{h_cond}
  -\ell \I \preceq \nabla^2 h(x) + \nabla^2 q(x) \preceq \ell \I.
  \end{equation}
  Let $\{x^t\}$ be generated from \eqref{FBSupdate} with $\beta \in (0,\frac{1}{\ell})$. Then the algorithm is a descent algorithm.
  Moreover, any cluster point $x^*$ of $\{x^t\}$, if exists, is a stationary point.
\end{theorem}
\begin{remark}
  For the algorithm to converge faster, intuitively, a larger step-size $\beta$ should be chosen; see also Table~\ref{table3}. Condition \eqref{h_cond} indicates that the ``concave" part of the smooth objective $h$ does not impose any restrictions on the choice of step-size. This could result in an $\ell$ smaller than the Lipschitz continuity modulus of $\nabla h(x)$, and hence allow a choice of a larger $\beta$. On the other hand, since the algorithm is a descent algorithm by Theorem~\ref{prop:prox}, the sequence generated from \eqref{FBSupdate} would be bounded under standard coerciveness assumptions on the objective function.
\end{remark}
\begin{proof}
  Notice from assumption that $\nabla(h + q)$ is Lipschitz continuous with Lipschitz continuity modulus at most $\ell$.
  Hence
  \begin{equation}\label{upperbd}
  (h+q)(x^{t+1})\le (h+q)(x^t) + \langle \nabla h(x^t) + \nabla q(x^t),x^{t+1}-x^t\rangle + \frac{\ell}{2}\|x^{t+1}-x^t\|^2.
  \end{equation}
  From this we see further that
  \begin{equation}\label{eq:descent}
    \begin{split}
      & h(x^{t+1}) + P(x^{t+1})  = (h + q)(x^{t+1}) + P(x^{t+1}) - q(x^{t+1})\\
      & \le  (h+q)(x^t) + \langle \nabla h(x^t) + \nabla q(x^t),x^{t+1}-x^t\rangle + \frac{\ell}{2}\|x^{t+1}-x^t\|^2 + P(x^{t+1}) - q(x^{t+1})\\
      & = h(x^t) + \langle \nabla h(x^t),x^{t+1}-x^t\rangle + \frac{\ell}{2}\|x^{t+1}-x^t\|^2 + P(x^{t+1}) \\
      &\phantom{ = } + q(x^t) + \langle \nabla q(x^t),x^{t+1}-x^t\rangle- q(x^{t+1})\\
      &\le h(x^t) + P(x^t) + \left(\frac{\ell}{2} - \frac{1}{2\beta}\right)\|x^{t+1}-x^t\|^2,
    \end{split}
  \end{equation}
  where the first inequality follows from \eqref{upperbd}, the last inequality follows from the definition of $x^{t+1}$ and the subdifferential inequality applied to the function $q$.
  Since $\beta\in (0,\frac{1}{\ell})$ implies $\frac{1}{2\beta}>\frac{\ell}{2}$, \eqref{eq:descent} shows that the algorithm is a descent algorithm.

  Rearranging terms in \eqref{eq:descent} and summing from $t=0$ to any $N-1 > 0$, we see further that
  \begin{equation*}
  \begin{split}
    \left(\frac{1}{2\beta} - \frac{\ell}{2}\right)\sum_{t=0}^{N-1}\|x^{t+1}-x^t\|^2 &
    \le h(x^0) + P(x^0) - h(x^N) - P(x^N).
  \end{split}
  \end{equation*}
  Now, let $x^*$ be a cluster point and take any convergent subsequence $\{x^{t_i}\}$ that converges to $x^*$. Taking limit on both sides of the above inequality
  along the convergent subsequence, one can see that $\lim\limits_{t \rightarrow \infty}\|x^{t+1}-x^t\|=0$.
  Finally, we wish to show that $\displaystyle \lim_{i \rightarrow \infty} P(x^{t_i+1})= P(x^*)$. To this end, note first that since $\lim\limits_{t \rightarrow \infty}\|x^{t+1}-x^t\|=0$, we also have $\lim\limits_{i \rightarrow \infty} x^{t_i+1} = x^*$. Then it follows from lower semicontinuity of $P$ that $\displaystyle \liminf_{i \rightarrow \infty} P(x^{t_i+1}) \ge P(x^*)$. On the other hand, from \eqref{FBSupdate}, we have
  \[
  \langle \nabla h(x^{t_i}),x^{t_i+1} - x^{t_i}\rangle + \frac{1}{2\beta}\|x^{t_i+1} - x^{t_i}\|^2 + P(x^{t_i+1}) \le \langle \nabla h(x^{t_i}),x^* - x^{t_i}\rangle + \frac{1}{2\beta}\|x^* - x^{t_i}\|^2 + P(x^*),
  \]
  which gives $\displaystyle \limsup_{i \rightarrow \infty} P(x^{t_i+1}) \le P(x^*)$. Hence, $\displaystyle \lim_{i \rightarrow \infty}P(x^{t_i+1})= P(x^*)$. Now, using this, $\displaystyle \lim_{t \rightarrow \infty}\|x^{t+1}-x^t\|=0$, \eqref{outersemi}
  and taking limit along the convergent subsequence in the following relation obtained from \eqref{FBSupdate}
  \begin{equation}\label{1stcondition}
    0\in \nabla h(x^t) + \frac{1}{\beta}(x^{t+1}-x^t) + \partial P(x^{t+1}),
  \end{equation}
  we see that the conclusion concerning stationary point holds.
\end{proof}

We illustrate the above theorem in the following examples.

\begin{example}
  Suppose that $h$ admits an explicit representation as a difference of two convex twice continuously differentiable functions $h = h_1 - h_2$, and that $h_1$ has a Lipschitz continuous gradient with modulus at most $L_1$. Then \eqref{h_cond} holds with $q = h_2$ and $\ell = L_1$. Hence, the step-size can be chosen from $(0,1/L_1)$.

  A concrete example of this kind is given by $h(x) = \frac12\langle x,Q x\rangle$, where $Q$ is a symmetric indefinite matrix.
  Then \eqref{h_cond} holds with $q(x) = -\frac{1}{2}\langle x,Q_-x\rangle$, where $Q_-$ is the projection of $Q$ onto the cone of nonpositive semidefinite matrices, and $\ell = \lambda_{\max}(Q) > 0$. The step-size $\beta$ can be chosen within the open interval $(0,1/\lambda_{\max}(Q))$.

  In the case when $h(x)$ is a concave quadratic, say, for example, $h(x) = -\frac12 \|\A x - b\|^2$ for some linear map $\A$, it is easy to see that \eqref{h_cond} holds with $q(x) = \frac12 \|\A x\|^2$ for {\em any} positive number $\ell$. Thus, step-size can be chosen to be any positive number.
\end{example}

\begin{example}
  Suppose that $h$ has a Lipschitz continuous gradient and it is known that all the eigenvalues of $\nabla^2h(x)$, for any $x$, lie in the interval $[-\lambda_2,\lambda_1]$ with $-\lambda_2 < 0 < \lambda_1$. If $\lambda_1 \ge \lambda_2$, it is clear that $\nabla h$ is Lipschitz continuous with modulus bounded by $\lambda_1$, and hence the step-size for the proximal gradient algorithm can be chosen from $(0,1/\lambda_1)$. On the other hand, if $\lambda_1 < \lambda_2$, then it is easy to see that \eqref{h_cond} holds with $q(x) = \frac{\lambda_2-\lambda_1}{4}\|x\|^2$ and $\ell = (\lambda_2 + \lambda_1)/2$. Hence,
  the step-size can be chosen from $(0,2/(\lambda_1 + \lambda_2))$.
\end{example}

We next comment on the convergence of the whole sequence. We consider the conditions {\bf H1} through {\bf H3} on \cite[Page~99]{AtBoSv13}. First, it is easy to see from \eqref{eq:descent} that {\bf H1} is satisfied with $a = \frac{1}{2\beta} - \frac{\ell}{2}$. Next, notice from \eqref{1stcondition} that if $w^{t+1}:= \nabla h(x^{t+1}) - \nabla h(x^t) - \frac{1}{\beta}(x^{t+1}-x^t)$, then $w^{t+1}\in \nabla h(x^{t+1}) + \partial P(x^{t+1})$. Moreover, from the definition of $w^{t+1}$, we have
\[
\|w^{t+1}\| \le \left( L + \frac{1}{\beta}\right)\|x^{t+1} - x^t\|
\]
for any $L \ge \sup\{\|\nabla^2h(x)\|:\; x\in \R^n\}$.
This shows that the condition {\bf H2} is satisfied with $b = L + \frac{1}{\beta}$. Finally, \cite[Remark~5.2]{AtBoSv13} shows that {\bf H3} is satisfied. Thus, we conclude from \cite[Theorem~2.9]{AtBoSv13} that if $h+P$ is a KL-function and a cluster point $x^*$ of the sequence $\{x^t\}$ exists, then the whole sequence converges to $x^*$.

A line-search strategy can also be incorporated to possibly speed up the above algorithm; see \cite{GHLYZ13}
for the case when $P$ is a continuous difference-of-convex function. The convergence analysis there can be directly adapted. The result of Theorem~\ref{prop:prox} concerning the interval of viable step-sizes can be used in designing the initial step-size for backtracking in the line-search procedure.


\section{Numerical simulations}\label{sec:num}

In this section, we perform numerical experiments to illustrate our algorithms. All codes are written in MATLAB.
All experiments are performed on a 32-bit desktop machine with an Intel$\circledR$ i7-3770 CPU (3.40 GHz) and a 4.00 GB RAM, equipped with MATLAB 7.13 (2011b).

\subsection{ADMM}

\paragraph{Minimizing constraints violation.} We consider the problem of finding the closest point to a given $\hat x\in \R^n$ that violates at most $r$ out of $m$ equations. The problem is presented as follows:
\begin{equation}\label{eq:cpv}
  \begin{array}{rl}
    \min\limits_x & \frac{1}{2}\|x - \hat x\|^2\\
    {\rm s.t.} & \|\M x - b\|_0 \le r,
  \end{array}
\end{equation}
where $\M \in \R^{m\times n}$ has full row rank, $b\in \R^m$, $n \ge m \ge r$. This can be seen as a special case of \eqref{P1} by taking $h(x) = \frac{1}{2}\|x - \hat x\|^2$ and $P(y)$ to be the indicator function of the set $\{y:\; \|y - b\|_0 \le r\}$, which is a proper closed function; here, $\|y\|_0$ is the $\ell_0$ norm that counts the number of nonzero entries in the vector $y$.

We apply the ADMM (i.e., proximal ADMM with $\phi = 0$) with parameters specified as in Example~\ref{example:5}, and pick $\beta = 1.01\cdot(2/\sigma)$ so that $\beta > 2/\sigma$. From Example~\ref{examplenew:2}, the sequence generated from the ADMM is always bounded and hence convergence of the sequence is guaranteed by Theorem~\ref{th:2}. We compare our model against the standard convex model with the $\ell_0$ norm replaced by the $\ell_1$ norm. This latter model is solved by SDPT3 (Version 4.0), called via CVX (Version 1.22), using default settings.

For the ADMM, we consider two initializations: setting all variables at the origin ($0$ init.), or setting $x^0$ to be the approximate solution $\tilde x$ obtained from solving the convex model, $y^0 = \M x^0$ and $z^0 = (\M\M^*)^{-1}\M(x^0 - \hat x)$ ($\ell_1$ init.). As discussed in Remark~\ref{rem1}, when $\tilde x$ is feasible for \eqref{eq:cpv}, this latter initialization satisfies the conditions in Theorem~\ref{thm:main}(ii).
We terminate the ADMM when the sum of successive changes is small, i.e., when
\begin{equation}\label{ADMMterm}
\frac{\|x^t - x^{t-1}\| + \|y^t - y^{t-1}\| + \|z^t - z^{t-1}\|}{\|x^t\| + \|y^t\| + \|z^t\| + 1} < 10^{-8}.
\end{equation}

In our experiments, we consider random instances. In particular, to guarantee that the problem \eqref{eq:cpv} is feasible for a fixed $r$, we generate the matrix $\M$ and the right hand side $b$ using the following MATLAB codes:
\begin{verbatim}
  M = randn(m,n);
  x_orig = randn(n,1);
  J = randperm(m);
  b = randn(m,1);
  b(J(1:m-r)) = M(J(1:m-r),:)*x_orig; % subsystem has a solution
\end{verbatim}
We then generate $\hat x$ with i.i.d. standard Gaussian entries.

We consider $n=1000$, $2000$, $3000$, $4000$ and $5000$, $m = 500$, $r = 100$, $200$ and $300$. We generate one random instance for each $(n,m,r)$ and solve \eqref{eq:cpv} and the corresponding $\ell_1$ relaxation.
The computational results are shown in Table~\ref{table2}, where we report the number of violated constraints (vio) by the approximate solution $x$ obtained, defined as $\#\{i:\;|(\M x - b)_i| > 10^{-4}\}$, and the distance from $\hat x$ (dist) defined as $\|x - \hat x\|$. We also report the number of iterations the ADMM takes, as well as the CPU time of both the ADMM initialized at the origin and SDPT3 called using CVX.\footnote{We include the preprocessing time by CVX in the CPU time.} We see that the model \eqref{eq:cpv} allows an explicit control on the number of violated constraints. In addition, comparing with the $\ell_1$ model, the $\ell_0$ model solved using the ADMM always gives a solution closer to $\hat x$. Finally, the solution obtained from the ADMM initialized from an approximate solution of the $\ell_1$ model can be slightly closer to $\hat x$ than the solution obtained from the zero initialization, depending on the particular problem instance.

\begin{table}[h!]
\caption{Computational results for perturbation with bounded number of violated equalities.}\label{table2}
\footnotesize
\begin{center}
\begin{tabular}{|ccc||cccc|ccc|ccc|}
\hline
\multicolumn{3}{|c||}{}&
\multicolumn{4}{c|}{$\ell_0$-ADMM ($0$ init.)}&
\multicolumn{3}{c|}{$\ell_1$-CVX}&
\multicolumn{3}{c|}{$\ell_0$-ADMM ($\ell_1$ init.)}
\\
$r$ & $n$& $\|x_{\rm orig} - \widehat x\|$&
{\rm iter}& {\rm CPU} & {\rm vio} & {\rm dist} & {\rm CPU} & {\rm vio} & {\rm dist}&
{\rm iter} & {\rm vio} & {\rm dist}
\\ \hline
 100 & 1000 & 4.70e+001 &        389 &  0.4 &        100 & 2.24e+001 & 10.1 &         13 & 3.25e+001 &        405 &        100 & 2.18e+001     \\
 100 & 2000 & 6.37e+001 &        158 &  0.4 &        100 & 2.05e+001 & 18.4 &          6 & 2.92e+001 &        150 &        100 & 1.89e+001     \\
 100 & 3000 & 7.72e+001 &        130 &  0.7 &        100 & 1.95e+001 & 27.7 &          8 & 2.97e+001 &        108 &        100 & 1.85e+001     \\
 100 & 4000 & 8.85e+001 &        101 &  0.8 &        100 & 2.01e+001 & 37.3 &          3 & 3.12e+001 &         95 &        100 & 1.89e+001     \\
 100 & 5000 & 1.00e+002 &         94 &  1.0 &        100 & 2.05e+001 & 49.7 &          3 & 2.96e+001 &         88 &        100 & 1.85e+001     \\ \hline
 200 & 1000 & 4.30e+001 &        518 &  0.4 &        200 & 1.50e+001 & 10.7 &         16 & 2.95e+001 &        577 &        200 & 1.38e+001     \\
 200 & 2000 & 6.35e+001 &        229 &  0.6 &        200 & 1.24e+001 & 21.1 &         12 & 2.91e+001 &        224 &        200 & 1.14e+001     \\
 200 & 3000 & 7.75e+001 &        146 &  0.8 &        200 & 1.22e+001 & 27.5 &          9 & 2.85e+001 &        136 &        200 & 1.21e+001     \\
 200 & 4000 & 9.14e+001 &        112 &  0.9 &        200 & 1.25e+001 & 37.2 &          5 & 2.78e+001 &        124 &        200 & 1.12e+001     \\
 200 & 5000 & 1.01e+002 &        113 &  1.2 &        200 & 1.17e+001 & 49.4 &          6 & 2.68e+001 &         97 &        200 & 1.06e+001     \\  \hline
 300 & 1000 & 4.65e+001 &        716 &  0.7 &        300 & 7.13e+000 &  9.2 &         22 & 2.81e+001 &        836 &        300 & 7.05e+000     \\
 300 & 2000 & 6.36e+001 &        219 &  0.6 &        300 & 5.95e+000 & 18.4 &         12 & 2.68e+001 &        232 &        300 & 6.33e+000     \\
 300 & 3000 & 7.88e+001 &        158 &  0.8 &        300 & 5.91e+000 & 29.3 &         12 & 2.58e+001 &        145 &        300 & 6.15e+000     \\
 300 & 4000 & 8.95e+001 &        142 &  1.1 &        300 & 5.61e+000 & 44.9 &         15 & 2.60e+001 &        140 &        300 & 6.27e+000     \\
 300 & 5000 & 1.01e+002 &        125 &  1.3 &        300 & 5.54e+000 & 49.4 &          7 & 2.73e+001 &        114 &        300 & 6.07e+000     \\ \hline
\end{tabular}
\end{center}
 \normalsize
\end{table}

\paragraph{Piecewise constant fitting.} We consider the problem of fitting a noisy signal $\hat x\in \R^n$ using a piecewise constant signal with $r$ pieces (see \cite[Example~9.16]{CaG14}):
\begin{equation}\label{eq:pcf}
  \begin{array}{rl}
    \min\limits_x & \frac{1}{2}\|x - \hat x\|^2\\
    {\rm s.t.} & \|\D x\|_0 \le r-1,
  \end{array}
\end{equation}
where $\D x$ is the $n-1$ dimensional vector whose $i$th entry is $x_{i+1} - x_i$. This is a special case of \eqref{P1} with $h(x) = \frac{1}{2}\|x - \hat x\|^2$ and $P(y)$ being the indicator function of the closed set $\{y:\; \|y\|_0 \le r-1\}$.

It is well known that $\D\D^*\succeq \sigma \I$ for $\sigma = 2(1 + \cos(\pi - \frac\pi{n}))$ \cite[Theorem~2.2]{KuScTs99}, which is close to zero when $n$ is large. Thus, the $\beta$ chosen as in the previous problem is large and can lead to slow convergence. As a heuristic, similarly as in \cite[Remark~2.1]{SunTohYang14}, we initialize $\beta$ as $\frac{1}{5n\sigma}$, and update $\beta$ as $\min\{1.0001\cdot\frac{2}{\sigma},2\beta\}$ when $\beta < \frac{2}{\sigma}$ and either $\|x^t\| > 10^{10}$ or $\|x^t - x^{t-1}\| > \frac{1000}{t}$. It is not hard to see that the sequence generated from the ADMM under this heuristic will still cluster at a stationary point of \eqref{eq:pcf}.

We initialize all variables at the origin and terminate when \eqref{ADMMterm} occurs. As a benchmark, we again look at the standard convex model with the $\ell_0$ norm replaced by the $\ell_1$ norm, solved by SDPT3 (Version 4.0), called via CVX (Version 1.22) using default settings.

In our experiments, we first generate a random piecewise constant signal and then perturb it with a Gaussian noise. Specifically, we use the following MATLAB codes:
\begin{verbatim}
  J = randperm(n-2) + 1; % from 2 to n-1, candidate break-points
  I = sort(J(1:r-1),'ascend'); % r-1 break-points
  x_orig = zeros(n,1);  x_orig(1:I(1)-1) = randn(1);
  for i = 1:r-2
    x_orig(I(i):I(i+1)-1) = randn(1);
  end
  x_orig(I(r-1):end) = randn(1);
  hatx = x_orig + tau*randn(n,1);
\end{verbatim}
We consider $n = 8000$, $10000$, $r = 50$, $100$ and $\tau = 0, 2.5\%$ and $5\%$. The computational results are shown in Table~\ref{table1.5}, where we present the number of iterations for our ADMM, the CPU time for both approaches in seconds,\footnote{We include the preprocessing time by CVX in the CPU time.} the cardinality ({\bf card}) of $\D x$ at the approximate solution $x^*$ for both methods, defined as $\#\{i:\;|(\D x)_i| > 10^{-4}\}$, and the recovery error $\frac{\|x - x_{\rm orig}\|}{\|x_{\rm orig}\|}$, where $x_{\rm orig}$ is the original noiseless piecewise constant signal. We see that the solution from our model always has the correct number of pieces, and is always closer to the original noiseless signal.

\vskip -0.3cm

\begin{table}[h!]
\caption{Computational results for perturbation with bounded number of violated equalities.}\label{table1.5}
\footnotesize
\begin{center}
\begin{tabular}{|ccc||cccc|ccc|}
\hline
\multicolumn{3}{|c||}{}&
\multicolumn{4}{c|}{$\ell_0$-ADMM}&
\multicolumn{3}{c|}{$\ell_1$-CVX}
\\
$\tau$ & $r$& $n$&
{\rm iter}& {\rm CPU} & {\rm card} & {\rm err} & {\rm CPU} & {\rm card} & {\rm err}
\\ \hline
 0.000 &   50 &       8000 &       4944 &  5.8 &         49 & 1.9e-008 &  2.7 &         49 & 2.4e-003       \\
 0.000 &   50 &      10000 &       4728 &  6.8 &         49 & 1.1e-008 &  2.2 &         46 & 5.5e-002       \\
 0.000 &  100 &       8000 &       5961 &  7.1 &         99 & 7.3e-007 &  2.0 &         97 & 1.8e-002       \\
 0.000 &  100 &      10000 &       7385 & 10.9 &         99 & 7.5e-007 &  2.6 &         90 & 5.9e-002       \\ \hline
 0.025 &   50 &       8000 &       4962 &  6.4 &         49 & 6.3e-003 &  2.0 &        118 & 5.9e-002       \\
 0.025 &   50 &      10000 &       6136 &  9.8 &         49 & 5.6e-003 &  2.3 &        106 & 6.8e-002       \\
 0.025 &  100 &       8000 &       5155 &  6.7 &         99 & 1.6e-002 &  1.9 &        164 & 7.3e-002       \\
 0.025 &  100 &      10000 &       5685 &  9.1 &         99 & 1.5e-002 &  2.3 &        206 & 6.4e-002       \\ \hline
 0.050 &   50 &       8000 &       4008 &  5.1 &         49 & 2.4e-002 &  1.7 &        137 & 5.5e-002       \\
 0.050 &   50 &      10000 &       5219 &  8.3 &         49 & 1.2e-002 &  2.3 &        134 & 3.1e-002       \\
 0.050 &  100 &       8000 &       3869 &  5.1 &         99 & 2.0e-002 &  1.7 &        229 & 5.9e-002       \\
 0.050 &  100 &      10000 &       4911 &  7.9 &         99 & 1.3e-002 &  2.6 &        237 & 4.0e-002       \\ \hline
\end{tabular}
\end{center}
 \normalsize
\end{table}
\vskip -0.3cm

Next, we present graphs to visualize
 the quality of the recovered signal via the above two methods: our ADMM method ($\ell_0$-ADMM) and the convex relaxation method ($\ell_1$-CVX).  To do this, we first generate a piecewise constant signal with $20$ pieces, and then perturb it with Gaussian noises with noise level $5\%$. The effect on recovering the original signal with $\ell_0$-ADMM method and the $\ell_1$-CVX method are shown in Figure~\ref{fig1}.
\vskip-0.5cm
\begin{figure}[h!]
\caption{Computational results piecewise constant fitting.}\label{fig1}
\begin{center}
\includegraphics[scale=0.54]{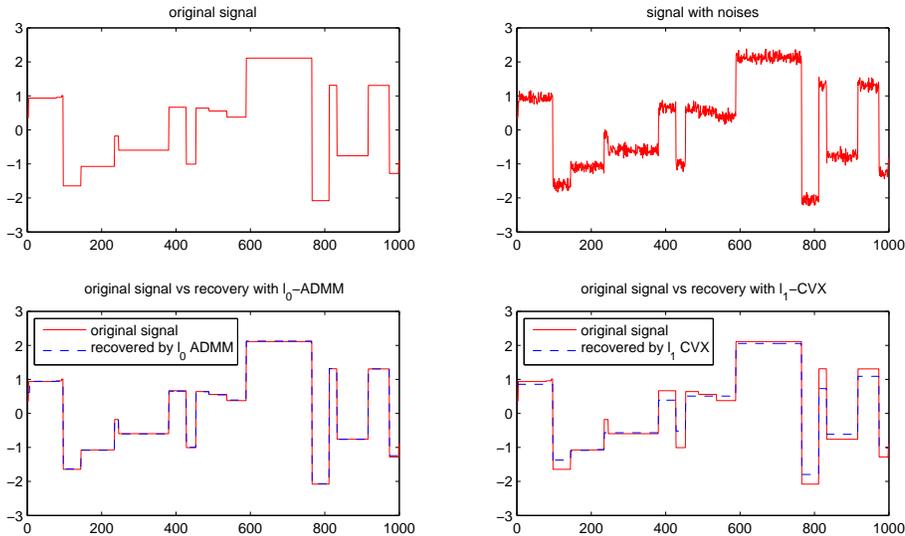}
\end{center}
\vskip-1.3cm
\end{figure}

\subsection{Proximal gradient algorithm}

In this section, we consider the following concave minimization problem:
\begin{equation}\label{sim1}
  \begin{array}{rl}
    \min\limits_x & -\frac{1}{2}\|\A x - b\|^2\\
    {\rm s.t.} & x \in {\cal C},
  \end{array}
\end{equation}
where ${\cal C}$ is a compact convex set whose projection is easy to compute, $\A \in \R^{m\times n}$ and $b\in \R^m$. We apply the proximal gradient algorithm and illustrate how the more flexible stepsize rule introduced via Theorem~\ref{prop:prox} affects the solution quality and the computational time. Specifically, we apply the proximal gradient algorithms with various step-size parameters $\beta > 0$. Since the objective in \eqref{sim1} is concave and $\cal C$ is compact, we see from Theorem~\ref{prop:prox} that for any $\beta > 0$, the sequence generated from the proximal gradient algorithm is bounded with cluster points being stationary points of \eqref{sim1}.

We initialize the algorithm at the origin and terminate when the change between successive iterates is small, i.e., when
\[
\frac{\|x^t - x^{t-1}\|}{\|x^t\| + 1} < 10^{-8}.
\]
We consider random instances. Specifically, for $m = 1000$ and each $n = 3000$, $4000$, $5000$ and $6000$, we generate a random matrix $\A\in \R^{m\times n}$ with i.i.d. standard Gaussian entries. We also generate $b\in \R^n$ with i.i.d. standard Gaussian entries.

The computational results are reported in Table~\ref{table3}, where we take $\cal C$ to be the unit $\ell_1$ norm ball for the first 4 rows, and the unit $\ell_\infty$ norm ball for the rest. We report the quantity $\lambda_{\max}(\A^*\A)$ for each of the random instances: the reciprocal of this quantity is typically used as an upper bound of the allowable step-size $\beta$ in the usual proximal gradient algorithm. We consider $\beta = 1/\lambda_{\max}(\A^*\A)$, $2/\lambda_{\max}(\A^*\A)$, $10/\lambda_{\max}(\A^*\A)$ and $50/\lambda_{\max}(\A^*\A)$, and report the terminating function value and number of iterations. We observe that the number of iterations is typically less when $\beta$ is larger. On the other hand, we can also observe that the terminating function values are not affected by the choice of step-size $\beta$ for the easier problems corresponding to the $\ell_1$ norm ball, but the solution quality concerning the $\ell_\infty$ norm ball does depend on the step-size $\beta$.

\begin{table}[h!]
\caption{Performance of the proximal gradient algorithm with varying $\beta$.}\label{table3}
\footnotesize
\begin{center}
\begin{tabular}{|cc||cc|cc|cc|cc|}
\hline
\multicolumn{2}{|c||}{$$}&
\multicolumn{2}{c|}{$\beta = 1/\lambda_{\max}(\A^*\A)$}&
\multicolumn{2}{c|}{$\beta = 2/\lambda_{\max}(\A^*\A)$}&
\multicolumn{2}{c|}{$\beta = 10/\lambda_{\max}(\A^*\A)$}&
\multicolumn{2}{c|}{$\beta = 50/\lambda_{\max}(\A^*\A)$}
\\
$n$& $\lambda_{\max}(\A^*\A)$& {\rm iter}& {\rm fval}& {\rm iter}& {\rm fval}& {\rm iter}& {\rm fval}& {\rm iter}& {\rm fval}
\\ \hline
  3000 & 7.41e+003 &         71 & -1.108e+003 &         44 & -1.108e+003 &          8 & -1.189e+003 &          4 & -1.189e+003     \\
  4000 & 8.97e+003 &         38 & -1.205e+003 &         21 & -1.205e+003 &          7 & -1.205e+003 &          4 & -1.205e+003     \\
  5000 & 1.04e+004 &         63 & -1.102e+003 &         34 & -1.102e+003 &         10 & -1.102e+003 &          5 & -1.102e+003     \\
  6000 & 1.19e+004 &         58 & -1.135e+003 &         30 & -1.135e+003 &          9 & -1.135e+003 &          4 & -1.135e+003     \\ \hline
  3000 & 7.44e+003 &        206 & -7.259e+006 &        207 & -7.180e+006 &         70 & -7.005e+006 &         44 & -6.829e+006     \\
  4000 & 8.96e+003 &        209 & -1.154e+007 &        175 & -1.148e+007 &        106 & -1.136e+007 &         55 & -1.122e+007     \\
  5000 & 1.05e+004 &        983 & -1.722e+007 &        244 & -1.709e+007 &        179 & -1.713e+007 &         56 & -1.694e+007     \\
  6000 & 1.18e+004 &       1068 & -2.318e+007 &        377 & -2.293e+007 &        166 & -2.292e+007 &         43 & -2.271e+007     \\  \hline
\end{tabular}
\end{center}
 \normalsize
\end{table}

\section{Conclusion and future directions}\label{sec:con}

In this paper, we study the proximal ADMM and the proximal gradient algorithm for solving problem \eqref{P1} with a general surjective $\M$ and $\M = \I$, respectively. We prove that any cluster point of the sequence generated from the algorithms gives a stationary point by assuming merely a specific choice of parameters and the existence of a cluster point. We also show that if the functions $h$ and $P$ are in addition semi-algebraic and the sequence generated by the ADMM (i.e., proximal ADMM with $\phi=0$) clusters, then the sequence is actually convergent. Furthermore, we give simple sufficient conditions for the boundedness of the sequence generated from the proximal ADMM.

One interesting future research direction would be to adapt other splitting methods for convex problems to solve \eqref{P1}, especially in the case when $\M$ is injective, and study their convergence properties.

\paragraph{Acknowledgement.} The second author would like to thank Ernie Esser and Gabriel Goh for enlightening discussions. The authors would also like to thank the anonymous referees for suggestions that help improve the manuscript.

\end{document}